\documentclass[reqno,11pt]{amsart}

\usepackage{amsmath,amsthm,amsfonts,amssymb,mathrsfs}
\usepackage[margin=1in]{geometry}
\usepackage[all,cmtip]{xy}
\usepackage[bookmarks]{hyperref}

\newcommand{\subgrp}[1]{\langle #1 \rangle}
\newcommand{\set}[1]{\left\{ #1 \right\}}

\newcommand{\mbf}[1]{\mathbf{#1}}
\newcommand{\wt}[1]{\widetilde{ #1}}
\newcommand{\wh}[1]{\widehat{ #1 }}

\DeclareMathOperator{\hght}{ht}
\DeclareMathOperator{\Ad}{Ad}
\newcommand{\Adl}{\Ad_\ell}
\newcommand{\Adr}{\Ad_r}

\DeclareMathOperator{\ch}{ch}
\DeclareMathOperator{\chr}{char}
\DeclareMathOperator{\coind}{coind}

\DeclareMathOperator{\cx}{cx}

\DeclareMathOperator{\Dist}{Dist}
\DeclareMathOperator{\End}{End}
\DeclareMathOperator{\Ext}{Ext}

\DeclareMathOperator{\gr}{gr}

\DeclareMathOperator{\opH}{H}
\DeclareMathOperator{\Hom}{Hom}
\DeclareMathOperator{\hy}{hy}

\DeclareMathOperator{\im}{im}
\DeclareMathOperator{\ind}{ind}
\DeclareMathOperator{\Lie}{Lie}
\DeclareMathOperator{\rad}{rad}
\DeclareMathOperator{\rank}{rank}
\DeclareMathOperator{\res}{res}
\DeclareMathOperator{\soc}{soc}
\DeclareMathOperator{\St}{St}
\newcommand{\Stprl}{\St_{p^r \ell}}
\DeclareMathOperator{\Tot}{Tot}

\DeclareMathAlphabet{\mathpzc}{OT1}{pzc}{m}{it}

\newcommand{\C}{\mathbb{C}}
\newcommand{\E}{\mathbb{E}}
\newcommand{\F}{\mathbb{F}}
\newcommand{\N}{\mathbb{N}}
\newcommand{\Q}{\mathbb{Q}}
\newcommand{\R}{\mathbb{R}}
\newcommand{\U}{\mathbb{U}}
\newcommand{\Z}{\mathbb{Z}}

\newcommand{\bfr}{\mathbf{r}}
\newcommand{\bfs}{\mathbf{s}}

\newcommand{\calF}{\mathcal{F}}

\newcommand{\calN}{\mathcal{N}}

\newcommand{\frakb}{\mathfrak{b}}
\newcommand{\g}{\mathfrak{g}}

\newcommand{\fraku}{\mathfrak{u}}

\newcommand{\scrA}{\mathscr{A}}

\newcommand{\scrZ}{\mathscr{Z}}

\newcommand{\upA}{\mathsf{A}}
\newcommand{\upB}{\mathsf{B}}
\newcommand{\upL}{\mathsf{L}}

\newcommand{\Bbul}{\upB_\bullet}
\newcommand{\Cbul}{C^\bullet}
\newcommand{\Qbul}{Q_\bullet}

\newcommand{\Ua}{U_\upA}
\newcommand{\Uagr}{U_\upA(G_r)}

\newcommand{\Uk}{U_k}
\newcommand{\Ukgr}{\Uk(G_r)}

\newcommand{\Uq}{\U_q}
\newcommand{\Uqg}{\U_q(\g)}
\newcommand{\Uqo}{\Uq^0}
\newcommand{\Uz}{U_\zeta}
\newcommand{\Uzb}{U_\zeta(B)}
\newcommand{\Uzbp}{U_\zeta(B^+)}
\newcommand{\Uzbr}{U_\zeta(B_r)}
\newcommand{\Uzbrp}{U_\zeta(B_r^+)}
\newcommand{\Uzbrt}{U_\zeta(B_rT)}
\newcommand{\Uzbrpt}{U_\zeta(B_r^+T)}
\newcommand{\Uzg}{U_\zeta(\g)}
\newcommand{\Uzgr}{U_\zeta(G_r)}
\newcommand{\Uzgrb}{U_\zeta(G_rB)}
\newcommand{\Uzgrbp}{U_\zeta(G_rB^+)}
\newcommand{\Uzgrt}{U_\zeta(G_rT)}
\newcommand{\UzN}{U_\zeta(N)}
\newcommand{\Uzo}{U_\zeta^0}

\newcommand{\Uztr}{U_\zeta(T_r)}
\newcommand{\Uzu}{U_\zeta(U)}

\newcommand{\Uzur}{U_\zeta(U_r)}
\newcommand{\Uzurp}{U_\zeta(U_r^+)}
\newcommand{\Uqxi}{U_{\Q(\xi)}}
\newcommand{\Uzxi}{U_{\Z[\xi]}}

\newcommand{\uz}{u_\zeta}
\newcommand{\uzb}{u_\zeta(\frakb)}

\newcommand{\uzg}{u_\zeta(\g)}

\newcommand{\uzu}{u_\zeta(\fraku)}

\newcommand{\uqxi}{u_{\Q(\xi)}}
\newcommand{\uzxi}{u_{\Z[\xi]}}

\newcommand{\whZr}{\wh{Z}_r}
\newcommand{\wtZr}{\wt{Z}_r}
\newcommand{\Lz}{L_\zeta}
\newcommand{\Lzr}{L_{\zeta,r}}
\newcommand{\whLzr}{\wh{L}_{\zeta,r}}
\newcommand{\wtLzr}{\wt{L}_{\zeta,r}}
\newcommand{\wtQzr}{\wt{Q}_{\zeta,r}}
\newcommand{\Qzr}{Q_{\zeta,r}}

\numberwithin{equation}{subsection}

\newtheorem{theorem}{Theorem}[subsection]
\newtheorem{proposition}[theorem]{Proposition}

\newtheorem{corollary}[theorem]{Corollary}
\newtheorem{lemma}[theorem]{Lemma}

\theoremstyle{definition}
\newtheorem{example}[theorem]{Example}

\newtheorem{definition}[theorem]{Definition}
\newtheorem{remark}[theorem]{Remark}
\newtheorem{remarks}[theorem]{Remarks}

\title[Representations and cohomology for Frobenius--Lusztig kernels]{Representations and cohomology for \\ Frobenius--Lusztig kernels}

\author{Christopher M.\ Drupieski}

\address{
Department of Mathematics \\
University of Georgia \\
Athens, GA~30602-7403}
\email{cdrup@math.uga.edu}

\thanks{The author was supported in part by NSF VIGRE grant DMS-0738586.}


\begin{document}

\begin{abstract}
Let $U_\zeta$ be the quantum group (Lusztig form) associated to the simple Lie algebra $\mathfrak{g}$, with parameter $\zeta$ specialized to an $\ell$-th root of unity in a field of characteristic $p>0$. In this paper we study certain finite-dimensional normal Hopf subalgebras $U_\zeta(G_r)$ of $U_\zeta$, called Frobenius--Lusztig kernels, which generalize the Frobenius kernels $G_r$ of an algebraic group $G$. When $r=0$, the algebras studied here reduce to the small quantum group introduced by Lusztig. We classify the irreducible $U_\zeta(G_r)$-modules and discuss their characters. We then study the cohomology rings for the Frobenius--Lusztig kernels and for certain nilpotent and Borel subalgebras corresponding to unipotent and Borel subgroups of $G$. We prove that the cohomology ring for the first Frobenius--Lusztig kernel is finitely-generated when $\g$ has type $A$ or $D$, and that the cohomology rings for the nilpotent and Borel subalgebras are finitely-generated in general.
\end{abstract}

\maketitle

\section{Introduction} \label{section:introduction}

\subsection{Background}

In the last twenty years, many deep connections have been discovered between the representation theory for algebraic groups and that for quantized enveloping algebras (quantum groups). For example, let $G$ be a semisimple, simply-connected algebraic group over an algebraically closed field of characteristic $p>0$, let $\g$ be the Lie algebra of $G$, and let $\Uz$ be the quantum group (Lusztig form) associated to $\g$, with parameter $\zeta \in \C$ an $\ell$-th root of unity. Let $G_1 \subset G$ be the first Frobenius kernel of $G$, and let $\uz \subset \Uz$ be Lusztig's small quantum group. A long outstanding problem has been to compute the characters of the irreducible rational $G$-modules. In their seminal paper \cite{Andersen:1994}, Andersen, Jantzen and Soergel proved that if $p$ is sufficiently large, then every irreducible $G_1$-module can be obtained via reduction mod $p$ from an irreducible $\uz$-module. In particular, if $p$ is sufficiently large, then the Lusztig character formula for the characters of irreducible rational $G$-modules can be deduced from the corresponding character formula for irreducible integrable $\Uz$-modules. Other deep connections between algebraic and quantum groups include the fact that if $p$ and $\ell$ are both greater than the Coxeter number of $\g$, then the cohomology rings for $G_1$ and $\uz$ are both isomorphic to the coordinate ring of the variety of nilpotent elements in $\g$ \cite{Andersen:1984,Friedlander:1986a,Ginzburg:1993}.

Though quantum groups can be defined with the parameter $\zeta$ taken to be in any field $k$, relatively little specific attention has been paid to the case when $k$ has characteristic $p > 0$. Just as the algebraic group $G$ in characteristic $p$ possesses a tower of normal subgroup schemes, $G_1 \subset G_2 \subset \cdots \subset G$, its Frobenius kernels, so too does a quantum group in characteristic $p$ possess a tower of normal Hopf subalgebras, $\uz = \Uz(G_0) \subset \Uz(G_1) \subset \Uz(G_2) \subset \cdots \subset \Uz$, which we call the Frobenius--Lusztig kernels of $\Uz$.

In this paper we study the representation theory and cohomology of the Frobenius--Lusztig kernels of a quantum group $\Uz$ defined over a field of characteristic $p$. Our results simultaneously generalize classical results on the Frobenius kernels of algebraic groups, as well as results of Andersen, Polo and Wen \cite{Andersen:1992a,Andersen:1992} for the small quantum group. It is our hope that the Frobenius--Lusztig kernels of a quantum group in characteristic $p$ might prove as useful to the study of representations and cohomology for $\Uz$ as have the Frobenius kernels to the study of representations and cohomology for the algebraic group $G$ and its finite subgroups.
 
\subsection{Organization}

The paper is organized as follows. In Section \ref{section:preliminaries} we establish notation and recall basic notions about quantum groups. Then in Section \ref{section:representationtheory} we establish basic facts about the representation theory of the Frobenius--Lusztig kernels of $\Uz$. We prove that $\Uzgr$ is a normal subalgebra of $\Uz$, and we classify the irreducible $\Uzgr$-modules, showing that each irreducible $\Uzgr$-module is the restriction of an irreducible $\Uz$-module with highest weight lying in a suitable restricted region. Section \ref{section:representationtheory} culminates with our proof that if $p \gg 0$, then the characters of the simple $\uz$-modules in characteristic $p$ are the same as in characteristic zero, hence are given by the Lusztig character formula for quantum groups \cite[II.H.12]{Jantzen:1996}. It is an interesting open problem to determine if there is a value of $p$ for which the characteristic $p$ and characteristic zero characters of simple $\uz$-modules do not agree.

Sections \ref{section:Hopfalgebrasoncohomology}--\ref{section:higherFLkernels} are devoted to studying the cohomology of the Frobenius--Lusztig kernels, specifically, the cohomology ring $\opH^\bullet(\Uzgr,k)$, and also the cohomology rings for the Borel and nilpotent subalgebras of $\Uzgr$ corresponding to a Borel subgroup $B$ of $G$ and its unipotent radical $U \subset B$. A major open problem in the study of finite-dimensional Hopf algebras is to determine whether the associated cohomology ring is finitely generated. Given a Hopf algebra $H$ over the field $k$, the cohomology ring $\opH^\bullet(H,k)$ is known to be finitely-generated if $H$ is cocommutative, by work of Friedlander and Suslin \cite{Friedlander:1997}, or if $H$ is pointed and $\chr(k) = 0$, by work of Mastnak, Pevtsova, Schauenburg and Witherspoon \cite{Mastnak:2010}. The higher Frobenius--Lusztig kernels of a quantum group in characteristic $p$ represent another infinite class of noncommutative, non-cocommutative finite-dimensional Hopf algebras, so it would be interesting to determine if their cohomology rings are finitely-generated as well.

We begin in Section \ref{section:Hopfalgebrasoncohomology} by describing a general theory for the actions of Hopf algebras on cohomology groups. Then in Section \ref{section:firstFLkernel} we are able to imitate the inductive approach of Friedlander and Parshall \cite{Friedlander:1986} to prove (with some mild restrictions on $\ell$, and assuming $\g$ to be of type $A$ or $D$) the finite-generation of the cohomology ring $\opH^\bullet(\Uz(G_1),k)$ for the first Frobenius--Lusztig kernel of $\Uz$. Unlike in the classical situation for algebraic groups, there is no canonical way to embed an arbitrary quantum group into one whose associated Lie algebra is of type $A$, so we are not able to prove finite-generation of $\opH^\bullet(\Uz(G_1),k)$ in general at this time. The inductive approach to studying the first Frobenius--Lusztig kernel requires calculating the cohomology ring $\opH^\bullet(\uz,k)$ for the small quantum group in characteristic $p$; this we do in Section \ref{subsection:uzgcohomology}. It is our opinion that the theory of Section \ref{section:Hopfalgebrasoncohomology} makes more transparent the Hopf algebra actions considered by Ginzburg and Kumar \cite{Ginzburg:1993} in their characteristic zero calculation of $\opH^\bullet(\uz,\C)$.

In Section \ref{section:higherFLkernels} we study, for all $r \geq 0$, the cohomology rings $\opH^\bullet(\Uzbr,k)$ and $\opH^\bullet(\Uzur,k)$ for the subalgebras of $\Uzgr$ corresponding to the Borel subgroup $B$ and its unipotent radical $U$. We prove that $\opH^\bullet(\Uzur,k)$ and $\opH^\bullet(\Uzbr,k)$ are finitely-generated as rings. A version of our results on $\opH^\bullet(\Uzur,k)$, appearing in the author's thesis \cite{Drupieski:2009}, was used by Feldvoss and Witherspoon \cite[Theorem 4.2]{Feldvoss:2010} as part of their proof that the principal block of the small quantum group is of wild representation type. Finally, in Section \ref{subsection:finitecomplexity} we provide some circumstantial evidence for the finite-generation of $\opH^\bullet(\Uzgr,k)$ in general. Specifically, we prove that the complexity of finite-dimensional $\Uzgr$-modules is finite and uniformly bounded.

\section{Preliminaries} \label{section:preliminaries}

\subsection{Quantized enveloping algebras} \label{subsection:QEAs}

Let $\Phi$ be a finite, indecomposable root system. Fix a set of simple roots $\Pi \subset \Phi$, and let $\Phi^+$ and $\Phi^-$ be the corresponding sets of positive and negative roots in $\Phi$. Write $W$ for the Weyl group of $\Phi$. It is generated by the set of simple reflections $\set{s_\alpha:\alpha \in \Pi}$. Let $\ell: W \rightarrow \N$ be the length function on $W$.

Let $\Z\Phi$ be the root lattice of $\Phi$. It spans a real vector space $\mathbb{E}$, possessing a positive definite, $W$-invariant inner product $(\cdot,\cdot)$, normalized so that $(\alpha,\alpha) = 2$ if $\alpha \in \Phi$ is a short root. Given $\alpha \in \Phi$, let $\alpha^\vee = 2\alpha/(\alpha,\alpha)$ be the dual root. Let $X = \set{\lambda \in \E: (\lambda,\alpha^\vee) \in \Z \; \forall \alpha \in \Phi}$ be the weight lattice of $\Phi$. It is spanned by the set of fundamental dominant weights $\set{\varpi_\alpha:\alpha \in \Pi}$, which are defined by the equations $(\varpi_\alpha,\beta^\vee) = \delta_{\alpha,\beta}$ ($\beta \in \Pi$). Let $X^+ = \set{\lambda \in X: (\lambda,\alpha^\vee) \geq 0 \; \forall \alpha \in \Phi^+}$ be the subset of dominant weights. Set $\rho = \frac{1}{2} \sum_{\alpha \in \Phi^+} \alpha = \sum_{\alpha \in \Pi} \varpi_\alpha$, and let $\alpha_0 \in \Phi^+$ be the highest short root. Then the Coxeter number $h$ of $\Phi$ is defined by $h = (\rho,\alpha_0^\vee)+1$.

Let $q$ be an indeterminate. The quantized enveloping algebra $\Uq$ is the $\Q(q)$-algebra defined by the generators $\set{E_\alpha,F_\alpha,K_\alpha,K_\alpha^{-1}:\alpha \in \Pi}$ and the relations in \cite[4.3]{Jantzen:1996}. The algebra $\Uq$ admits the structure of a Hopf algebra, with comultiplication $\Delta$, counit $\varepsilon$, and antipode $S$ defined in \cite[4.8]{Jantzen:1996}.

Set $\upA = \Z[q,q^{-1}]$. For $a \in \Z$, put $[a] = \frac{q^a-q^{-a}}{q-q^{-1}} \in \upA$, and for $n \in \N$, define $[n]^! = [n] [n-1] \cdots [2] [1]$. For $a \in \Z$, define $[a]_\alpha$ to be the image in $\upA$ of $[a]$ under the ring homomorphism $\upA \rightarrow \upA$ mapping $q \mapsto q_\alpha:=q^{(\alpha,\alpha)/2}$. Let $\Ua$ be the $\upA$-subalgebra of $\Uq$ generated by
\[
\set{E_\alpha^{(n)},F_\alpha^{(n)},K_\alpha,K_\alpha^{-1}:\alpha \in \Pi, n \in \N},
\]
where the divided powers $E_\alpha^{(n)},F_\alpha^{(n)} \in \Uq$ are defined by $E_\alpha^{(n)} = E_\alpha^n/[n]_\alpha^!$ and $F_\alpha^{(n)} = F_\alpha^n/[n]_\alpha^!$. Then $\Ua$ is a free $\upA$-module and an $\upA$-form for $\Uq$ \cite{Lusztig:1990a}. It is also a Hopf subalgebra of $\Uq$. For any $\upA$-algebra $\Gamma$, write $U_\Gamma$ for the algebra $\Ua \otimes_{\upA} \Gamma$. We follow the usual convention of writing the superscripts $+$, $-$ and $0$ to denote the positive, negative, and toral subalgebras of $\Uq$, of the $\upA$-form $\Ua$, and of the specializations $U_\Gamma$. Then, for example, $\Uq^+$ is the subalgebra of $\Uq$ generated by the set $\{E_\alpha: \alpha \in \Pi \}$. There exists an involutory $\Q(q)$-algebra automorphism $\omega$ of $\Uq$ satisfying $\omega(E_\alpha) = F_\alpha$, $\omega(F_\alpha) = E_\alpha$, and $\omega(K_\alpha) = K_\alpha^{-1}$ ($\alpha \in \Pi$). It descends to an automorphism of $\Ua$.

\subsection{Frobenius--Lusztig kernels} \label{subsection:FLkernels}

Let $k$ be a field of characteristic $p \neq 2$, and $p \neq 3$ if $\Phi$ has type $G_2$. Let $\ell \in \N$ be an odd positive integer, with $\ell$ coprime to $p$, and also $\ell$ coprime to 3 if $\Phi$ has type $G_2$. (There should be no confusion between the use of $\ell$ as an integer, and the use of $\ell$ for the length function on $W$.) Fix a primitive $\ell$-th root of unity $\zeta \in k$. Then $k$ is naturally an $\upA$-module under the specialization $q \mapsto \zeta$. Define $\Uz$ to be the quotient of $\Uk = \Ua \otimes_{\upA} k$ by the two-sided ideal $\subgrp{K_\alpha^\ell \otimes 1 - 1 \otimes 1 : \alpha \in \Pi}$. In the language of \cite{Andersen:1992a,Andersen:1992,Andersen:1991}, every $\Uz$-module is a $\Uk$-module of type 1. Conversely, every $\Uk$-module of type 1 is automatically a $\Uz$-module. By abuse of notation, we denote from now on the generators of $\Ua$ as well as their images in $\Uk$ and $\Uz$ by the same symbols.

The elements $\set{E_\alpha,F_\alpha,K_\alpha: \alpha \in \Pi}$ of $\Uk$ generate a finite-dimensional Hopf subalgebra $u_k$ of $\Uk$. We denote the image of $u_k$ in $\Uz$ by $\uz$, and call this latter algebra the small quantum group. It is a normal Hopf subalgebra of $\Uz$ (see Section \ref{subsection:normality}), and the quotient $\Uz//\uz$ is isomorphic as a Hopf algebra to $\Dist(G)$, the algebra of distributions on the simple, simply-connected algebraic group $G$ over $k$ with root system $\Phi$. (For details on $\Dist(G)$, see \cite[II.1.12]{Jantzen:2003}.) The algebra $\Dist(G)$ is also known as the hyperalgebra of $G$, and denoted by $\hy(G)$. The quotient map $F_\zeta: \Uz \rightarrow \Dist(G)$ was constructed by Lusztig \cite[\S 8]{Lusztig:1990a}, and is called the quantum Frobenius morphism. For this reason, the algebra $\uz$ is also called the Frobenius--Lusztig kernel of $\Uz$. Given a $\Dist(G)$-module $V$, write $V^{[1]}$ for $V$ considered as a $\Uz$-module via the morphism $F_\zeta: \Uz \rightarrow \Dist(G)$.

Set $\g = \Lie(G)$, the Lie algebra of $G$. If we wish to emphasize the dependence of the algebras $\Uq$, $\Uz$, $\uz$, etc., on the root system $\Phi$ of $\g$, then we write $\Uq = \Uqg$, $\Uz = \Uzg$, $\uz = \uzg$, etc.

Fix $r \in \N$, and suppose $p = \chr(k) > 0$. Define $\Uzgr$ to be the subalgebra of $\Uz$ generated by
\begin{equation} \label{eq:kernelgenerators}
\set{E_\alpha,E_\alpha^{(p^i \ell)}, F_\alpha, F_\alpha^{(p^i \ell)}, K_\alpha : \alpha \in \Pi, 0 \leq i \leq r-1}.
\end{equation}
Then $\Uzgr$ is a finite-dimensional Hopf subalgebra of $\Uz$, and $F_\zeta(\Uzgr) = \Dist(G_r)$, the algebra of distributions on the $r$-th Frobenius kernel $G_r$ of $G$. We call $\Uzgr$ the $r$-th Frobenius--Lusztig kernel of $\Uz$. By definition, the zeroth Frobenius--Lusztig kernel, $\Uz(G_0)$, is just the small quantum group $\uz$. We refer to the $\Uzgr$ with $r \geq 1$ as the higher Frobenius--Lusztig kernels of $\Uz$.  The higher Frobenius--Lusztig kernels of $\Uz$ are defined only if $p = \chr(k) > 0$. Indeed, if $\chr(k) = 0$ and $r \geq 1$, then the subalgebra of $\Uz$ generated by \eqref{eq:kernelgenerators} is all of $\Uz$ \cite[Proposition 3.2]{Lusztig:1989}.

\subsection{Braid group automorphisms and integral bases} \label{subsection:braidgroupauto}

Let $w_0 = s_{\beta_1} \cdots s_{\beta_N}$ be a reduced expression for the longest word $w_0 \in W$. For each $\gamma \in \Phi^+$, there exist root vectors $E_\gamma \in \Uq^+$ and $F_\gamma \in \Uq^-$, defined in terms of certain braid group operators on $\Uq$, and depending on the chosen reduced expression for $w_0$ \cite[Appendix]{Lusztig:1990a}.

For $1 \leq i \leq N$, set $w_i = s_{\beta_1} \cdots s_{\beta_{i-1}}$ (so $w_1 = 1$), and set $\gamma_i = w_i(\beta_i)$. Then $\Phi^+ = \set{\gamma_1,\ldots,\gamma_N}$, and $\upA$-bases for $\Ua^-$ and $\Ua^+$, respectively, are given by the collections of monomials
\begin{align}
& \set{F^{(\bfr)} = F_{\gamma_1}^{(r_1)} \cdots F_{\gamma_N}^{(r_N)}: \bfr = (r_1,\ldots,r_N) \in \N^N}, \text{ and} \label{eq:Fdivpowerbasis} \\
& \set{E^{(\bfs)} = E_{\gamma_1}^{(s_1)} \cdots E_{\gamma_N}^{(s_N)}: \bfs = (s_1,\ldots,s_N) \in \N^N} \label{eq:Edivpowerbasis}
\end{align}
where $E_\gamma^{(n)}$ and $F_\gamma^{(n)}$ for arbitrary $\gamma \in \Phi^+$ are defined in \cite[\S 5.1]{Lusztig:1990a}. An $\upA$-basis for the $\upA$-form $\Ua^0$ of $\Uqo$ is described in \cite[Theorem 6.7]{Lusztig:1990a}.

Multiplication induces $\upA$-module isomorphisms $\Ua^+ \otimes_\upA \Ua^0 \otimes_\upA \Ua^- \cong \Ua \cong \Ua^- \otimes_\upA \Ua^0 \otimes_\upA \Ua^+$ (triangular decomposition). We thus obtain for any $\upA$-algebra $\Gamma$ a $\Gamma$-basis for $U_\Gamma$. The $\upA$-bases for $\Ua^+$ and $\Ua^-$ project onto $k$-bases for $U_\zeta^+$ and $U_\zeta^-$, and there exist similar vector space isomorphisms $\Uz^+ \otimes \Uzo \otimes \Uz^- \cong \Uz \cong \Uz^- \otimes \Uzo \otimes \Uz^+$.

\subsection{Distinguished subalgebras} \label{subsection:subalgebras}

Fix a maximal torus $T \subset G$ such that $\Phi$ is the root system of $T$ in $G$. Let $B \subset G$ be the Borel subgroup of $G$ containing $T$ and corresponding to $\Phi^-$. Let $U$ be the unipotent radical of $B$. To the subgroup schemes $U$, $B$, $U_r$, $B_r$, $T_r$, $B_rT$, $G_rT$, and $G_rB$ of $G$, we associate certain distinguished subalgebras of $U_\zeta$. Define $\Uzu = \Uz^-$, $\Uzb = \Uz^- \Uz^0$, $\Uzur = \Uz^- \cap \Uzgr$, $\Uzbr = \Uzb \cap \Uzgr$, $\Uztr = \Uzo \cap \Uzgr$, $\Uzbrt = \Uzbr \Uzo$, $\Uzgrt = \Uzgr \Uzo$, and $\Uzgrb = \Uzgr \Uzb$. Then, for example, $\Uzur$ admits a basis consisting of all monomials in \eqref{eq:Fdivpowerbasis} with $0 \leq a_i < p^r \ell$ for all $1 \leq i \leq N$. Replacing $B$ by its opposite Borel subgroup $B^+$, we similarly define the opposite subalgebras $U_\zeta(U_r^+)$, $U_\zeta(B_r^+)$, etc., of $\Uz$. The automorphism $\omega$ induces isomorphisms $\Uz^- \stackrel{\sim}{\rightarrow} \Uz^+$, $\Uzur \stackrel{\sim}{\rightarrow} U_\zeta(U_r^+)$, etc.

If $r > 0$, then the distinguished subalgebras defined in the previous paragraph generalize those studied in \cite{Andersen:1992a,Andersen:1992}. When $r=0$, set $\uzg = \Uzgr$, $\uzu = \Uzur$, and $\uzb = \Uzbr$. Here $\fraku = \Lie(U)$ and $\frakb = \Lie(B)$. These notations are meant to emphasize the similarity between the algebras $\uzg$, $\uzu$, $\uzb$ and the restricted enveloping algebras $u(\g)$, $u(\fraku)$, $u(\frakb)$ of the $p$-restricted Lie algebras $\g$, $\fraku$, $\frakb$.

\section{Representation theory} \label{section:representationtheory}

\subsection{Normality} \label{subsection:normality}

Let $A$ be a Hopf algebra, and let $B \subseteq A$ be a subalgebra. Then $B$ is called a normal subalgebra of $A$ if $B$ is closed under the left and right adjoint actions of $A$ on itself. Write $B_+$ for the augmentation ideal of $B$. If $B$ is normal in $A$, then $AB_+ = B_+A$, i.e., the left and right ideals in $A$ generated by $B_+$ are equal \cite[Lemma 3.4.2]{Montgomery:1993}. For $B$ normal in $A$, put $A//B = A/(AB_+)$. If $B$ is a normal Hopf subalgebra of $A$, then $A//B$ inherits from $A$ the structure of Hopf algebra. In this section we show that the Frobenius--Lusztig kernels are normal in $\Uz$.

\begin{proposition} \label{proposition:adjointstable}
Let $r \geq 0$. Then $\Uzgr$ is normal in $\Uz$.
\end{proposition}

\begin{proof}
First, if $\chr(k) = 0$ (so also $r=0$), then $\uz$ is normal in $\Uz$ by \cite[Proposition 5.3]{Lin:1993}. So assume that $p = \chr(k) > 0$. Consider \eqref{eq:kernelgenerators} as a subset of the algebra $\Uk$, and let $\Ukgr$ denote the subalgebra it generates. Then it suffices to show that $\Ukgr$ is normal in $\Uk$. Recall that for a Hopf algebra $H$ with bijective antipode $S$, the left and right adjoint actions are related via the equation
\[
\Adr(h) = S \circ \Adl( S^{-1}(h)) \circ S^{-1} \qquad (h \in H).
\]
So it suffices even to show that $\Ukgr$ is stable under the left adjoint action of $\Uk$ on itself.

Given $n \in \N$, let $\phi_n$ be the $n$-th cyclotomic polynomial. By our assumptions on $\ell$ and $p$, the product $p^r\ell$ is odd and coprime to 3 if $\Phi$ has type $G_2$, so it satisfies the same assumptions as $\ell$. The map $\upA = \Z[q,q^{-1}] \rightarrow k$ sending $q \mapsto \zeta$ maps $\phi_\ell$ to zero. Let $\varphi$ be Euler's totient function. According to \cite{Guerrier:1968}, the polynomial $\phi_{p^r\ell}$ factors over $\F_p$ as $(\phi_\ell)^{\varphi(p^r)}$. Then the map $\upA \rightarrow k$ factors through the quotient $\upA/(\phi_{p^r\ell})$. Let $\xi \in \C$ be a primitive $p^r\ell$-th root of unity. Then the quotient $\upA/(\phi_{p^r\ell})$ identifies with the subring $\Z[\xi]$ of the cyclotomic field $\Q(\xi)$.

Let $\Uagr$ be the subalgebra of $\Ua$ generated by \eqref{eq:kernelgenerators}, considered now as a subset of $\Ua$. Then $\Uagr \otimes_\upA k = \Ukgr$. Also, $\Adl(\Ua)(\Uagr) \subseteq \Ua$, because $\Ua$ is a Hopf subalgebra of $\Uq$. To show that $\Adl(\Uk)(\Ukgr) \subseteq \Ukgr$, it suffices to show that $(\Adl(\Ua)(\Uagr)) \otimes_\upA k \subseteq \Ukgr$. Observe that $\Uagr \otimes_\upA \Q(\xi) = \uqxi$, the subalgebra of $\Uqxi$ that projects onto the small quantum group associated to the $p^r\ell$-th root of unity $\xi \in \Q(\xi)$. Then by the case $\chr(k)=0$, $(\Adl(\Ua)(\Uagr)) \otimes_\upA \Z[\xi] \subseteq \uqxi = \Uagr \otimes_\upA \Q[\xi]$. Since base change from $\upA$ to $k$ factors through $\Z[\xi]$, we get
\begin{align*}
(\Adl(\Ua)(\Uagr)) \otimes k &= \left( \Adl(\Ua)(\Uagr) \otimes_\upA \Z[\xi] \right) \otimes_{\Z[\xi]} k \\
&\subseteq (\Uagr \otimes_\upA \Z[\xi]) \otimes_{\Z[\xi]} k = \Ukgr. \qedhere
\end{align*}
\end{proof}

\begin{corollary} \label{corollary:uzbstable}
The algebra $\uzb$ is stable under the right adjoint action of $\Uzb$.
\end{corollary}

\begin{proof}
The algebra $\Uzb$ is a Hopf algebra containing $\uzb$, and $\uzb \subset \uzg$. Then $\Adr(\Uzb)(\uzb) \subseteq \Uzb \cap \uzg = \uzb$.
\end{proof}

\subsection{Local finiteness}

Let $N$ be one of the distinguished subgroup schemes of $G$ identified in Section \ref{subsection:subalgebras}. Recall the definition, due to Andersen, Polo and Wen \cite{Andersen:1992a,Andersen:1992,Andersen:1991}, of the category of integrable $\UzN$-modules. If $N$ is a finite subgroup scheme of $G$, then every $\UzN$-module is integrable. Otherwise, a $\UzN$-module $V$ is integrable if it satisfies the following two conditions:
\begin{enumerate}
\item If $T \subset N$, then $V$ admits a weight space decomposition for $\Uzo$ (in the sense of \cite[\S 1.2]{Andersen:1991}).
\item Given $\alpha \in \Phi^+$, let $U_\alpha \subset B$ and $U_\alpha^+ \subset B^+$ be the corresponding root subgroups of $G$. Let $v \in V$. If $U_\alpha \subset N$, then $F_\alpha^{(n)}.v = 0$ for all $n \gg 0$. If $U_\alpha^+ \subset N$, then $E_\alpha^{(n)}.v = 0$ for all $n \gg 0$.
\end{enumerate}

Our goal now is to study the integrable representation theory of the distinguished  subalgebras $\Uzgr$, $\Uzgrt$, and $\Uzgrb$ defined in Section~\ref{subsection:subalgebras}. Specifically, we wish to characterize the simple integrable modules for these algebras. When $r=0$, our results reproduce those of Andersen, Polo and Wen \cite{Andersen:1992a,Andersen:1992} for the small quantum group. We begin our investigation with the following lemma, which implies that the simple modules we wish to study are all finite-dimensional. First some notation: Given $\lambda \in X$, there exists a one-dimensional integrable $\Uzb$-module (resp.\ $\Uzbp$-module) of $\Uzo$-weight $\lambda$, with trivial $\Uz^-$ (resp.\ $\Uz^+$)-action; denote it by the symbol $\lambda$.

\begin{lemma} \label{lemma:locallyfinite}
Let $V$ be an integrable $\UzN$-module. Then $V$ is locally finite, i.e., every finitely-generated submodule of $V$ is contained in a finite-dimensional submodule of $V$.
\end{lemma}

\begin{proof}
Let $0 \neq v \in V$. To prove the lemma, it suffices to show that $v$ generates a finite-dimensional $\UzN$-submodule of $V$. If $N$ is a finite subgroup scheme of $G$, then $\UzN$ is finite-dimensional, in which case the result is obvious. If $V$ is a $\Uz(NT)$-module, then $v = \sum_{\mu \in X} v_\mu$, a finite sum of weight vectors, and $\Uz(NT).v \subseteq \sum \Uz(N)\Uzo.v_\mu = \sum \UzN.v_\mu$, which is again finite-dimensional.

Suppose $N = U$. For each $\alpha \in \Pi$, set $n_\alpha = \max \{n \in \N: F_\alpha^{(n)}.v \neq 0\}$, and set $\lambda = \sum_{\alpha \in \Pi} n_\alpha \varpi_\alpha$. Then $M(\lambda):= \Uz \otimes_{\Uzbp} \lambda$ is the Verma module of highest weight $\lambda$, generated by the vector $x_\lambda := 1 \otimes \lambda$. As a vector space and as a module for $\Uz^-$, $M(\lambda) \cong \Uz^-$. Now define $J(\lambda)$ to be the left ideal of $\Uz^-$ generated by $\{ F_\alpha^{(n)}: \alpha \in \Pi, n > n_\alpha \}$, and define $N(\lambda)$ to be the $\Uz^-$-submodule $J(\lambda).x_\lambda$ of $M(\lambda)$. Since the action of $\Uz^-$ on $v$ factors through the quotient $\Uz^-/J(\lambda)$, it suffices to show that the quotient is finite-dimensional. As a $\Uz^-$-module, $\Uz^-/J(\lambda) \cong M(\lambda)/N(\lambda)$. Arguing as in the proof of \cite[Proposition 1.20]{Andersen:1991}, we see that $N(\lambda)$ is a $\Uz$-submodule of $M(\lambda)$, and that the quotient $M(\lambda)/N(\lambda)$ is an integrable $\Uz$-module. Then by \cite[Proposition 1.7]{Andersen:1992a}, the Weyl group $W$ acts on the weights of $M(\lambda)/N(\lambda)$. Continuing as in the proof of \cite[Proposition 1.20]{Andersen:1991}, we conclude that $M(\lambda)/N(\lambda)$, and hence also $\Uz^-.v$, is finite-dimensional. This proves the lemma for the case $N=U$, from which the case $N=G_rB$ also follows. (Write $\Uzgrb = \Uzgrt \Uzu$.)

Finally, suppose $N=G$. Again writing $v = \sum v_\mu$, we have $\Uz.v \subseteq \sum \Uz.v_\mu$, so it suffices to show that each $\Uz.v_\mu$ is finite-dimensional, i.e., we may assume that $v$ is a weight vector. Then $\Uz.v = \Uz^+ \Uzo \Uz^- .v = \Uz^+ \Uz^- \Uzo .v = \Uz^+ \Uz^-.v = \Uz^+(\Uz^-.v)$. By the case $N=U$, the space $V':=\Uz^-.v$ is finite-dimensional. Then by the case $N=U^+$ (which is completely analogous to the case $N=U$), the space $\Uz.v = \Uz^+.V'$ is as well.
\end{proof}

\begin{remark}
Conversely, every locally finite $\Uz$-module is integrable \cite[Theorem A3.7]{Du:1994}.
\end{remark}

\subsection{Baby Verma modules}

Let $\ind_{U_1}^{U_2}(-) = \opH^0(U_2/U_1,-)$ be the induction functor for quantized enveloping algebras defined in \cite{Andersen:1992a,Andersen:1992,Andersen:1991}. Given $\lambda \in X$, define the integrable modules
\begin{align*}
\whZr'(\lambda) &= \ind_{\Uzb}^{\Uzgrb} \lambda, \\
\wtZr'(\lambda) &= \ind_{\Uzbrt}^{\Uzgrt} \lambda, \quad \text{and} \\
Z_r'(\lambda) &= \ind_{\Uzbr}^{\Uzgr} \lambda.
\end{align*}
Then $\whZr'(\lambda)|_{\Uzgrt} \cong \wtZr'(\lambda)$ and $\wtZr'(\lambda)|_{\Uzgr} \cong Z_r'(\lambda)$ (cf.\ \cite[\S 1.1--1.3]{Andersen:1992}). Now let $\coind_K^H (-) = H \otimes_K -$ be the usual tensor induction functor for Hopf algebras, and define the integrable modules
\begin{align*}
\whZr(\lambda) &= \coind_{\Uzbp}^{\Uzgrbp} \lambda, \\
\wtZr(\lambda) &= \coind_{\Uzbrpt}^{\Uzgrt} \lambda, \quad \text{and} \\
Z_r(\lambda) &= \coind_{\Uzbrp}^{\Uzgr} \lambda.
\end{align*}
Then $\whZr(\lambda)|_{\Uzgrt} \cong \wtZr(\lambda)$ and $\wtZr(\lambda)|_{\Uzgr} \cong Z_r(\lambda)$. As modules for the Hopf algebra $\Uzo$, $\whZr'(\lambda) \cong \Hom_k(\Uzurp,k) \otimes \lambda$, and $\whZr(\lambda) \cong \Uzur \otimes \lambda$. Here $\Uzur$ and $\Uzurp$ are viewed as $\Uzo$-modules via the adjoint action of $\Uzo$.

\begin{lemma} \textup{\cite[Lemma 3.6]{Drupieski:2011a}} \label{lemma:projectivecoverinjectivehull}
Let $\lambda \in X$.
\begin{enumerate}
\item In the category of integrable $\Uzbrt$-modules, $\wtZr(\lambda)$ is the projective cover of $\lambda$ and the injective hull of $\lambda - 2(p^r\ell-1)\rho$.
\item In the category of integrable $\Uzbrpt$-modules, $\wtZr'(\lambda)$ is the projective cover of $\lambda - 2(p^r\ell-1)\rho$ and the injective hull of $\lambda$.
\end{enumerate}
Statements (1) and (2) are also true if the modules $\wtZr(\lambda)$ and $\wtZr'(\lambda)$ are replaced by $Z_r(\lambda)$ and $Z_r'(\lambda)$, and if the algebras $\Uzbrt$ and $\Uzbrpt$ are replaced by $\Uzbr$ and $\Uzbrp$.
\end{lemma}

Let $H$ be a Hopf algebra with antipode $S$, and let $V$ be a left $H$-module. Recall that the dual space $V^* = \Hom_k(V,k)$ is made into an $H$-module by setting $(h.f)(v) = f(S(h).v)$ for all $f \in V^*$, $h \in H$, and $v \in V$. For the Hopf algebras considered in this paper, if $V$ is finite-dimensional, then $V^{**} = (V^*)^*$ is naturally (though not canonically) isomorphic to $V$ as an $H$-module \cite[5.3(3)]{Jantzen:1996}.

\begin{lemma} \textup{\cite[Lemma 3.7]{Drupieski:2011a}} \label{lemma:Zdual}
Let $\lambda \in X$. Then
\begin{align*}
\whZr'(\lambda)^* &\cong \whZr'(2(p^r\ell-1)\rho - \lambda) \qquad \text{and} \\
\whZr(\lambda)^* &\cong \whZr(2(p^r\ell-1)\rho - \lambda).
\end{align*}
\end{lemma}

\subsection{Simple modules}

Given $\lambda \in X$, define
\begin{equation} \label{eq:definesimples}
\begin{split}
\whLzr(\lambda) &= \soc_{\Uzgrb} \whZr'(\lambda), \\
\wtLzr(\lambda) &= \soc_{\Uzgrt} \wtZr'(\lambda), \quad \text{and} \\
\Lzr(\lambda) &= \soc_{\Uzgr} Z_r'(\lambda).
\end{split}
\end{equation}

In the context of quantized enveloping algebras, it is not a priori clear that there should be any inclusion relations among the modules in \eqref{eq:definesimples}. We will eventually prove that the modules in \eqref{eq:definesimples} are in fact equal. First we require the following theorem, which follows from arguments completely analogous to those used in the context of algebraic groups (cf.\ \cite[I.9.6]{Jantzen:2003}).

\begin{theorem} \label{theorem:simplemodules}
Let $N \in \set{G_r,G_rT,G_rB}$, and let $\upL(\lambda)$ be the corresponding module in \eqref{eq:definesimples}.
\begin{enumerate}
\item $\upL(\lambda)$ is a simple $\UzN$-module.
\item $\upL(\lambda)^{\Uzurp} \cong \lambda$ as a $\Uztr$-module, and $\End_{\UzN}(\upL(\lambda)) \cong k$.
\item If $T \subset N$, then $\upL(\lambda)^{\Uzurp} = \upL(\lambda)_\lambda$, and each weight $\mu$ of $\upL(\lambda)$ satisfies $\mu \leq \lambda$.
\item For all $\lambda,\mu \in X$, $\upL(\lambda + p^r\ell\mu) \cong \upL(\lambda) \otimes p^r\ell\mu$ as a $\UzN$-module.
\item If $N \in \set{G_rT,G_rB}$, then the $\upL(\lambda)$ for $\lambda \in X$ form a complete set of pairwise non-isomorphic simple integrable $\UzN$-modules. If $N = G_r$, then the $\upL(\lambda)$ with
\[
\lambda \in X_{p^r\ell}:= \set{\mu \in X^+: 0 \leq (\mu,\alpha^\vee) < p^r\ell \; \forall \alpha \in \Pi}
\]
form a complete set of pairwise non-isomorphic simple integrable $\UzN$-modules.
\item There exist isomorphisms of $\UzN$-modules
\begin{gather}
\upL(2(p^r\ell-1)\rho - \lambda)^* \cong \whZr'(\lambda)/\rad_{\UzN} \whZr'(\lambda) \label{eq:sizone} \\
\upL(\lambda) \cong \whZr(\lambda)/\rad_{\UzN} \whZr(\lambda) \label{eq:sixtwo} \\
\upL(2(p^r\ell-1)\rho - \lambda)^* \cong \soc_{\UzN} \whZr(\lambda). \label{eq:sixthree}
\end{gather}
\end{enumerate}
\end{theorem}

Recall that the simple integrable $\Uz$-modules are parametrized by their highest dominant weights \cite[Proposition 1.6]{Andersen:1992a}. Given $\mu \in X^+$, let $\Lz(\mu)$ be the simple integrable $\Uz$-module of highest weight $\mu$, and let $L(\mu)$ be the simple rational $G$-module of highest weight $\mu$. Now let $\lambda \in X^+$, and write $\lambda = \lambda^0 + \ell \lambda^1$ with $\lambda^0 \in X_\ell$ and $\lambda^1 \in X^+$. Then by \cite[Theorem 1.10]{Andersen:1992a}, there exists a $\Uz$-module isomorphism
\begin{equation} \label{eq:tensorproducttheorem}
\Lz(\lambda) \cong \Lz(\lambda^0) \otimes L(\lambda^1)^{[1]}.
\end{equation}
The restriction of $\Lz(\lambda^0)$ to $\uz$ is simple by \cite[Theorem 1.9]{Andersen:1992a}, while $L(\lambda^1)$ is simple for $G$ (equivalently, for $\Dist(G)$).

\begin{lemma} \label{lemma:zeroonhighestweight}
Let $\lambda \in X_\ell$, and let $0 \neq v \in \Lz(\lambda)_\lambda$. Then $F_\alpha^{(n)}.v = 0$ for all $\alpha \in \Pi$ and $n \geq \ell$.
\end{lemma}

\begin{proof}
Let $\xi \in \C$ be a primitive $\ell$-th root of unity. Let $L_{\xi}(\lambda)$ be the integrable type 1 simple $\Uqxi$-module of highest weight $\lambda$ (i.e., the simple integrable $U_\xi$-module of highest weight $\lambda$). Fix a highest weight vector $x \in L_\xi(\lambda)_\lambda$. Then $F_\alpha^{(\ell)}.x = 0$ for all $\alpha \in \Pi$ by \cite[Proposition~7.1]{Lusztig:1989}, and consequently $F_\alpha^{(n)}.x = 0$ for all $\alpha \in \Pi$ and $n \geq \ell$ by \cite[3.2(c)]{Lusztig:1989}.

Now let $\phi_\ell$ be the $\ell$-th cyclotomic polynomial. Recall that the map $\upA = \Z[q,q^{-1}] \rightarrow k$ sending $q \mapsto \zeta$ factors through the quotient $\upA/(\phi_\ell) \cong \Z[\xi]$. The ring $\Uzxi = U_\upA \otimes_\upA \Z[\xi]$ is a subring of $\Uqxi$. Let $V'$ be the $\Uzxi$-submodule of $L_\xi(\lambda)$ generated by $x$. Set $V = V' \otimes_{\Z[\xi]} k$. Then $V$ is an integrable $U_k$-module of type 1, i.e., $V$ is an integrable $\Uz$-module. The module $V$ need not be simple for $\Uz$, though it does have $\Lz(\lambda)$ as a simple quotient because $\dim V_\lambda = 1$, and because all other weights $\mu$ of $V$ satisfy $\mu \leq \lambda$. Furthermore, the image of $x$ in $V$ projects onto the highest weight vector of $\Lz(\lambda)$. Since for all $\alpha \in \Pi$ and $n \geq \ell$ the equality $F_\alpha^{(n)}.x = 0$ holds in $L_\xi(\lambda)$, it must then hold in $V'$, and hence also in $V$ and in its simple quotient $\Lz(\lambda)$.
\end{proof}

\begin{theorem} \label{theorem:simplerestrictsimple}
Let $\lambda \in X_{p^r\ell}$. Then $\Lz(\lambda)$ is simple as a module for $\Uzgr$.
\end{theorem}

\begin{proof}
Write $\lambda = \lambda^0 + \ell \lambda^1$ with $\lambda^0 \in X_\ell$ and $\lambda^1 \in X_{p^r}$. Arguing as in the proof of \cite[Theorem 7.4]{Lusztig:1989}, one uses Lemma \ref{lemma:zeroonhighestweight} to show that $\Lz(\lambda) \cong \Lz(\lambda^0) \otimes L(\lambda^1)^{[1]}$ is generated as a $\Uzgr$-module by a highest weight vector $0 \neq x \in \Lz(\lambda)_\lambda$. Then, continuing as in \cite{Lusztig:1989}, one uses the simplicity of $\Lz(\lambda^0)$ for $\uz$ and the simplicity of $L(\lambda^1)$ for $\Dist(G_r) \cong \Uzgr//\uz$ to show that every $\Uzgr$-submodule of $\Lz(\lambda)$ contains $x$, and hence $\Lz(\lambda)$ is simple for $\Uzgr$.
\end{proof}

\begin{corollary} \label{corollary:simplescoincide}
Let $\lambda \in X$. The three submodules of $\whZr'(\lambda)$ defined in \eqref{eq:definesimples} coincide.
\end{corollary}

\begin{proof}
Since any $\lambda \in X$ can be written uniquely as $\lambda = \lambda' + p^r\ell \mu$ with $\lambda' \in X_{p^r\ell}$ and $\mu \in X$, it suffices by Theorem \ref{theorem:simplemodules}(4) to prove that the three submodules of $\whZr'(\lambda)$ coincide in the special case $\lambda \in X_{p^r\ell}$. So assume $\lambda \in X_{p^r\ell}$. By Theorem \ref{theorem:simplerestrictsimple}, $\Lz(\lambda)$ is simple as a $\Uzgr$-module, hence also as a $\Uzgrb$-module and as a $\Uzgrt$-module. Lemma \ref{lemma:zeroonhighestweight} and \eqref{eq:tensorproducttheorem} imply that the set of $\Uzurp$-invariants in $\Lz(\lambda)$ is precisely $\Lz(\lambda)_\lambda$. Then Theorem \ref{theorem:simplemodules}(2) implies $\Lz(\lambda) \cong \whLzr(\lambda)$ as a $\Uzgrb$-module, and hence that $\whLzr(\lambda) = \wtLzr(\lambda) = \Lzr(\lambda)$.
\end{proof}

\begin{remark}
Corollary \ref{corollary:simplescoincide} and Theorem \ref{theorem:simplemodules}(2) imply that every simple integrable $\Uzgrt$-module lifts uniquely to a simple $\Uzgrb$-module, and every simple integrable $\Uzgrb$-module is isomorphic to exactly one simple integrable $\Uzgrt$-module thus extended. By symmetry, the corresponding statement for $\Uzgrbp$ is also true.
\end{remark}

\subsection{Injective modules}

Given $\lambda \in X$, let $\wtQzr(\lambda)$ denote the injective hull of $\wtLzr(\lambda)$ in the category of integrable $\Uzgrt$-modules, and let $\Qzr(\lambda)$ denote the injective hull of $\Lzr(\lambda)$ in category of integrable $\Uzgr$-modules. By \cite[Lemma 3.3]{Drupieski:2011a}, $\wtQzr(\lambda)$ is also injective as a $\Uzgr$-module. Then arguing as in \cite[II.11.3]{Jantzen:2003}, we get that $\wtQzr(\lambda) \cong \Qzr(\lambda)$ as $\Uzgr$-modules. Using the results in \cite[\S 3.1--3.2]{Drupieski:2011a}, the proof of the following theorem now only requires a routine translation to the present context of the argument in \cite[II.11.4]{Jantzen:2003}, where the corresponding result for Frobenius kernels of algebraic groups is proved.

\begin{proposition} \textup{(cf.\ \cite[Proposition II.11.4]{Jantzen:2003})} \label{proposition:humphreysfiltration}
Let $\lambda \in X$. The $\Uzgrt$-module $\wtQzr(\lambda)$ admits filtrations $0 = M_0 \subset M_1 \subset \cdots \subset M_s = \wtQzr(\lambda)$ and $0 = M_0' \subset M_1' \subset \cdots \subset M_s' = \wtQzr(\lambda)$ such that each factor has the form $M_i/M_{i-1} \cong \wtZr(\lambda_i)$ resp.\ $M_i'/M_{i-1}' \cong \wtZr'(\lambda_i')$ with $\lambda,\lambda_i' \in X$. For each $\mu \in X$, the number of $i$ with $\lambda_i = \mu$ resp.\ with $\lambda_i' = \mu$ is equal to $[\wtZr(\mu):\wtLzr(\lambda)] = [\wtZr'(\mu):\wtLzr(\lambda)]$, the composition factor multiplicity of $\wtLzr(\lambda)$ in $\wtZr(\mu)$ resp.\ $\wtZr'(\mu)$.
\end{proposition}

\begin{corollary}
Let $\lambda \in X$. Then $\wtQzr(\lambda)$ is the projective cover of $\wtLzr(\lambda)$ in the category of integralbe $\Uzgrt$-modules, and $\Qzr(\lambda)$ is the projective cover of $\Lzr(\lambda)$ in the category of integrable $\Uzgr$-modules.
\end{corollary}

\begin{proof}
This follows from Proposition \ref{proposition:humphreysfiltration}, from \cite[Proposition 3.9]{Drupieski:2011a}, and from the argument in \cite[II.11.5]{Jantzen:2003}.
\end{proof}

Set $\Stprl = \Lz((p^r\ell-1)\rho)$. We call this module the $r$-th Steinberg module for $\Uz$. When $r=0$, we refer to $\St_\ell$ simply as the Steinberg module. The Weyl group $W$ acts on the weights of integrable $\Uz$-modules by \cite[Proposition 1.7]{Andersen:1992a}, so $w_0((p^r\ell-1)\rho) = -(p^r\ell-1)\rho$ is the lowest weight of $\Stprl$. Then $(\Stprl)^*$ is a simple integrable $\Uz$-module of highest weight $(p^r\ell-1)\rho$, hence is isomorphic to $\Stprl$, i.e., $\Stprl$ is self-dual. Now (\ref{eq:sizone}--\ref{eq:sixthree}) imply that there exist $\Uzgrt$-module isomorphisms
\begin{equation} \label{eq:SteinbergZiso}
\Stprl \cong \whZr'((p^r\ell-1)\rho) \cong \whZr((p^r\ell-1)\rho).
\end{equation}

\begin{corollary} \label{corollary:Steinbergprojective}
The $r$-th Steinberg module $\Stprl = \Lz((p^r\ell-1)\rho)$ is injective and projective as an integrable $\Uzgrt$-module and as a $\Uzgr$-module.
\end{corollary}

\begin{proof}
Apply Proposition \ref{proposition:humphreysfiltration} and \cite[Lemma 3.3]{Drupieski:2011a}.
\end{proof}

\subsection{Characters of simple modules}

Let $\set{e(\mu):\mu \in X}$ be the canonical basis for the group ring $k[X]$ of the weight lattice $X$ (an additive group). Given a finite-dimensional $\Uz$-module (resp.\ $G$-module) $M$, the formal character of $M$ is defined by $\ch M = \sum_{\mu \in X} (\dim M_\mu) e(\mu)$.

Let $\lambda \in X^+$, and write $\lambda = \lambda^0 + \lambda^1$ with $\lambda^0 \in X_\ell$ and $\lambda^1 \in X^+$. To compute the formal character of the simple $\Uz$-module $\Lz(\lambda)$, it suffices by \eqref{eq:tensorproducttheorem} to compute the formal characters of $\Lz(\lambda^0)$ and $L(\lambda^1)^{[1]}$. If $p:=\chr(k) = 0$, then $\ch L(\lambda^1)$ is given by the Weyl character formula, and if $\ell > h$, then $\ch \Lz(\lambda^0)$ can be computed by the Lusztig character formula \cite[II.H.12]{Jantzen:2003}.

If $p > 0$, then much less is certain. By relating quantum groups in characteristic zero to algebraic groups in characteristic $p > 0$, Andersen, Jantzen and Soergel \cite{Andersen:1994} have shown that for each root system $\Phi$, there exists an unknown bound $n(\Phi)$, depending only on $\Phi$, such that if $p > n(\Phi)$, then $\ch L(\lambda^1)$ can also be computed by the Lusztig character formula. Unfortunately, no effective lower bound for $n(\Phi)$ is known, though Fiebig \cite{Fiebig:2008} has recently computed a lower bound that is explicit but very large.

Now to compute the characters of the simple $\Uz$-modules when $\zeta \in k$ and $p = \chr(k) \gg 0$, it remains to compute $\ch \Lz(\lambda)$ for $\lambda \in X_\ell$. Recall the setup for the proof of Lemma \ref{lemma:zeroonhighestweight}: We have $\xi \in \C$ a primitive $\ell$-th root of unity, and $\Uqxi = \Ua \otimes_\upA \Q(\xi)$. The module $L_\xi(\lambda)$ is the integrable type 1 simple $\Uqxi$-module of highest weight $\lambda \in X_\ell$, and $\ch L_\xi(\lambda)$ can be computed by the Lusztig character formula.

\begin{theorem}
Let $\lambda \in X_\ell$. There exists an integer $N(\Phi)$, depending only on the root system $\Phi$, such that if $p := \chr(k) > N(\Phi)$, then $\ch \Lz(\lambda) = \ch L_\xi(\lambda)$. In particular, if $\ell > h$ and $p > N(\Phi)$, then $\ch \Lz(\lambda)$ is given by the Lusztig character formula.
\end{theorem}

\begin{proof}
By the proof of Lemma \ref{lemma:zeroonhighestweight}, $\dim \Lz(\lambda)_\mu \leq \dim L_\xi(\lambda)_\mu$ for all $\mu \in X$. Then to prove the theorem, it suffices to show that there exists a integer $N(\Phi)$, depending only on the root system $\Phi$ (i.e., not depending on the integer $\ell$ or on the choice of primitive $\ell$-th root of unity $\zeta \in k$) such that if $p > N(\Phi)$, then $\dim \Lz(\lambda) = \dim L_\xi(\lambda)$.

Let $\uzxi$ be the subalgebra of $\Uzxi = \Ua \otimes_\upA \Z[\xi]$ generated by $\set{E_\alpha,F_\alpha,K_\alpha:\alpha \in \Pi}$. Then $\uzxi$ and $\Uzxi$ are naturally subalgebras of $\uqxi$ and $\Uqxi$, respectively. As remarked in the proof of Lemma \ref{lemma:zeroonhighestweight}, the map $\upA \rightarrow k$ sending $q \mapsto \zeta$ factors through a map $\upA/(\phi_\ell) \cong \Z[\xi] \rightarrow k$. Then $\uzxi \otimes_{\Z[\xi]} k \cong u_k$.

The simple modules for $\uqxi$ and for $u_k$ are each parametrized by the same finite set, namely, the cartesian product $X_\ell \times (\Z/2\Z)^n$; cf.\ \cite[\S 1]{Andersen:1991}. Given $\lambda \in X_\ell$ and the identity element $e \in (\Z/2\Z)^n$, the simple module parametrized by $(\lambda,e)$ is just $L_\xi(\lambda)$ (resp.\ $\Lz(\lambda)$). Given $\sigma \in (\Z/2\Z)^n$, there exists a one-dimensional $\uqxi$-module (resp.\ $u_k$-module), also denoted $\sigma$, such that the simple module parametrized by $(\lambda,\sigma)$ is $L_\xi(\lambda) \otimes \sigma$ (resp.\ $\Lz(\lambda) \otimes \sigma$).

Let $L_1,\ldots,L_m$ and $L_1',\ldots,L_m'$ be representatives for the isomorphism classes of distinct simple $\uqxi$-modules (resp.\ $u_k$-modules). Since $\dim \Lz(\lambda) \leq \dim L_\xi(\lambda)$ for all $\lambda \in X_\ell$, we may assume $\dim L_i' \leq \dim L_i$ for all $1 \leq i \leq m$. For $1 \leq i \leq m$, let $Q_i$ be the $\uqxi$-projective hull of $L_i$, and let $Q_i'$ be the $u_k$-projective hull of $L_i'$. Since $\End_{\uqxi}(L_i) \cong \Q(\xi)$ (resp.\ $\End_{u_k}(L_i') \cong k$) by Theorem \ref{theorem:simplemodules}(2), we get by standard results for finite-dimensional algebras that the left regular modules decompose as
\[
\uqxi \cong \bigoplus_{i=1}^m (Q_i)^{\oplus \dim L_i} \qquad \text{and} \qquad u_k \cong \bigoplus_{i=1}^m (Q_i')^{\oplus \dim L_i'}.
\]
Write the above direct sum decomposition for $\uqxi$ as $\uqxi \cong P_1 \oplus \cdots \oplus P_s$, where $P_i \cong Q_1$ for $1 \leq i \leq \dim L_1$, $P_i \cong Q_2$ for $(\dim L_1+1) \leq i \leq \dim L_2$, and so on. Let $S \subset \uqxi$ be an ordered basis for $\uqxi$ such that the first $(\dim P_1)$ vectors in $S$ are a basis for $P_1$, the second $(\dim P_2)$ vectors in $S$ are a basis for $P_2$, and so on. There exists $N \in \N$ such that $S \subset \frac{1}{N} \uzxi$. Set $u' = \uzxi \otimes_{\Z[\xi]} \Z[\xi,1/N]$. Then the first $(\dim P_1)$ vectors in $S$ span a $u'$-submodule of $u'$, the second $(\dim P_2)$ vectors in $S$ span a $u'$-submodule of $u'$, and so on. 

Suppose that $S$ is chosen so as to make $N$ as small as possible. In this case, put $N(\Phi) = N$. Now suppose $p > N$. Then the map $\Z[\xi] \rightarrow k$ extends to a map $\Z[\xi,1/N] \rightarrow k$, and $u_k \cong u' \otimes_{\Z[\xi,1/N]} k$. It follows that the left regular module $u_k$ decomposes as $u_k \cong P_1' \oplus P_2' \oplus \cdots \oplus P_s'$, for some $u_k$-submodules $P_1',P_2',\ldots,P_s'$ of $u_k$ with  $\dim P_i' = \dim P_i$. By the Krull-Schmidt theorem, we must have $s \leq \sum_{i=1}^m \dim L_i'$. But $\sum_{i=1}^m \dim L_i' \leq \sum_{i=1}^m \dim L_i = s$. Then $\dim L_i = \dim L_i'$ for all $1 \leq i \leq m$. In particular, we must have $\dim \Lz(\lambda) = \dim L_\xi(\lambda)$ for all $\lambda \in X_\ell$.
\end{proof}

\section{Hopf algebra actions on cohomology} \label{section:Hopfalgebrasoncohomology}

\subsection{Cohomology of normal subalgebras} \label{subsection:cohomologyofnormalsubalgebras}

Let $A$ be an arbitrary augmented algebra over $k$, and let $B \subseteq A$ be a subalgebra. Generalizing the definition of a normal subalgebra given in Section \ref{subsection:normality}, we say that $B$ is (left) normal in $A$ if $B_+A \subseteq AB_+$. If $B$ is normal in $A$, put $A//B = A/(AB_+)$. Suppose $B$ is normal in $A$. Then the space of invariants $V^B \cong \Hom_B(k,V)$ is an $A$-submodule of $V$, and the map $-^B:V \mapsto V^B$ is an endofunctor on the category of $A$-modules.

The cohomology $\opH^n(B,W)$ of $B$ with coefficients in the $B$-module $W$ is defined by $\opH^n(B,W)=\Ext_B^n(k,W)=R^n(\Hom_B(k,-))(W)$. In this context, $\Hom_B(k,-)$ is considered as a functor from the category of $B$-modules to the category of $k$-vector spaces. The right derived functors of $\Hom_B(k,-)$ are defined in terms of $B$-injective resolutions, whereas the right derived functors of $-^B$ are defined in terms of $A$-injective resolutions. The following lemma gives a sufficient condition for $R^n(-^B)(V)$ and $\opH^n(B,V)$ to be isomorphic as $k$-vector spaces.

\begin{lemma} \label{lemma:injective} \textup{\cite[Lemma I.4.3]{Barnes:1985}}
Every injective $A$-module is injective for $B$ if and only if $A$ is flat as a right $B$-module.
\end{lemma}

As a consequence, one gets:

\begin{lemma} \textup{\cite[I.5]{Barnes:1985}} \label{lemma:uniqueaction}
Suppose that $A$ is right $B$-flat. Then for each $A$-module $V$, there exists a unique natural extension of the action of $A$ on $V^B$ to an action of $A$ on $\opH^\bullet(B,V)$. This action of $A$ on $\opH^\bullet(B,V)$ factors through the quotient $A//B$.
\end{lemma}

\subsection{Compatible actions for Hopf algebras}

Our next goal is to investigate the actions of Hopf algebras on the cohomology groups $\opH^\bullet(B,V)$. When $A$ is itself a Hopf algebra, this will give a new description for the action of $A$ on $\opH^\bullet(B,V)$. First recall the notion of an $H$-module algebra.

\begin{definition} \label{definition:Hmodulealgebra}
Let $H$ be a Hopf algebra. An algebra $A$ is an \emph{$H$-module algebra} if
\begin{enumerate}
	\item $A$ is an $H$-module,
	\item Multiplication $A \otimes A \rightarrow A$ is an $H$-module homomorphism, and
	\item $H$ acts trivially on $1_A \in A$.
\end{enumerate}
Additionally, if $A$ is augmented over $k$, with augmentation map $\varepsilon: A \rightarrow k$, we assume for all $a \in A$ and $h \in H$ that $\varepsilon(a \cdot h) = \varepsilon(a) \varepsilon(h)$.
\end{definition}

Any Hopf algebra is an $H$-module algebra over itself via the left and right adjoint actions \cite[Example 4.1.9]{Montgomery:1993}.

\begin{definition} \label{definition:compatiblestructure}
Let $H$ be a Hopf algebra, $A$ a right $H$-module algebra, and $V$ a left $A$-module that is simultaneously a left $H$-module. Given $h \in H$, write $\Delta(h) = \sum h_{(1)} \otimes h_{(2)}$. Then we say that the $A$- and $H$-module structures on $V$ are \emph{compatible} if for all $v \in V$, $a \in A$ and $h \in H$, we have $a.(h.v) = \sum h_{(1)}.((a \cdot h_{(2)}).v)$.
\end{definition}

\begin{example} \label{example:trivialcompatible}
The $A$- and $H$-module structures on the trivial module $k$ are compatible.
\end{example}

\begin{lemma} \label{lemma:stabilizesinvariants}
Let $A$ be a right $H$-module algebra, and let $V$ be a left $A$-module with compatible left $H$-module structure. Then $V^A$ is an $H$-submodule of $V$.
\end{lemma}

\begin{proof}
Let $v \in V^A$, and let $h \in H$. Then for all $a \in A$,
\begin{align*}
a.(h.v) &= \sum h_{(1)}. \left( (a \cdot h_{(2)}).v \right) = \sum \varepsilon(a)\varepsilon(h_{(2)})h_{(1)}.v = \varepsilon(a) h.v. \qedhere
\end{align*}
\end{proof}

\subsection{Actions on the bar resolution} \label{subsection:actionsonbarresolution}

Let $A$ be an augmented algebra over $k$, and let $B$ be a normal subalgebra of $A$. The left bar resolution $\Bbul(B) = B \otimes B_+^{\otimes \bullet}$ of $B$ is the chain complex with differential $d_n: \upB_n(B) \rightarrow \upB_{n-1}(B)$ defined by $d_n = \sum_{i=0}^{n-1} (-1)^i (1^{\otimes i} \otimes m \otimes 1^{\otimes n-i-1})$, where $m: B \otimes B \rightarrow B$ is the multiplication in $B$. Given a $B$-module $W$, set $\Cbul(B,W) = \Hom_B(\Bbul(B),W)$. Then $\opH^n(B,W)$ is the cohomology of the cochain complex $\Cbul(B,W)$.

Suppose $A$ is a right $H$-module algebra, and that $B$ is an $H$-submodule of $A$. Then the right action of $H$ on $B$ extends diagonally to an action of $H$ on $\Bbul(B)$, making $\Bbul(B)$ a complex of right $H$-modules. Now let $M$ be a right $H$-module and let $N$ be a left $H$-module. Then $\Hom_k(M,N)$ is made a left $H$-module by setting (for all $h \in H$, $f \in \Hom_k(M,N)$, and $m \in M$)
\begin{equation} \label{eq:Hdiagonalaction}
(h.f)(m) = \sum h_{(1)}.f(m \cdot h_{(2)}).
\end{equation}

\begin{theorem} \label{theorem:Hadjointaction}
Let $A$ be an augmented algebra over $k$, and let $B$ be a normal subalgebra of $A$. Assume that $A$ is a right $H$-module algebra, and that $B$ is an $H$-submodule of $A$. Let $V$ be a left $A$-module with compatible $H$-module structure. Then \eqref{eq:Hdiagonalaction} defines a left $H$-module structure on $\Cbul(B,V)$ such that $\Cbul(B,V)$ is a complex of $H$-modules. The left action of $H$ on $\Cbul(B,V)$ induces a left action of $H$ on $\opH^\bullet(B,V)$.
\end{theorem}

\begin{proof}
The left action of $H$ on $\Hom_k(\Bbul(B),V)$ stabilizes the subspace $\Cbul(B,V)$ of $B$-module homomorphisms and commutes with the differential of $\Cbul(B,V)$, because the $A$- and $H$-module structures on $V$ are compatible.
\end{proof}

\begin{definition} \label{definition:adjointaction}
We call the left action of $H$ on $\opH^\bullet(B,V)$ defined in Theorem \ref{theorem:Hadjointaction} the adjoint action of $H$ on $\opH^\bullet(B,V)$.
\end{definition}

The cup product $\cup$ defines a ring structure on $\opH^\bullet(B,k)$. Given cocycles $f \in C^n(B,k)$ and $g \in C^m(B,k)$, write $[f] \in \opH^n(B,k)$ and $[g] \in \opH^m(B,k)$ for the corresponding cohomology classes. Then the cup product $[f] \cup [g] \in \opH^{n+m}(B,k)$ is defined by $[f] \cup [g] = [f \cup g]$, where $f \cup g \in C^{n+m}(B,k)$ is defined by $(f \cup g)([b_1|\ldots|b_{n+m}]) = f([b_1|\cdots|b_n]) g([b_{n+1}|\cdots|b_{n+m}])$.

\begin{lemma} \label{lemma:cohomologyHmodulealgebra}
Let $A,B,H$ be as in Theorem \ref{theorem:Hadjointaction}. The adjoint action of $H$ on the cohomology ring $\opH^\bullet(B,k)$ makes $\opH^\bullet(B,k)$ a left $H$-module algebra.
\end{lemma}

\begin{proof}
The product $\cup$ on $\Cbul(B,k)$ is a homomorphism of $H$-modules.
\end{proof}

\subsection{Hopf algebra actions via injective resolutions} \label{subsection:actionsviainjectives}

Let $H$ be a Hopf algebra, $A$ a right $H$-module algebra, $B$ a normal subalgebra of $A$ stable under the action of $H$, and $V$ a left $A$-module with compatible $H$-action. So far we have described the adjoint action of $H$ on $\opH^\bullet(B,V)$ in terms of a $B$-projective resolution of $k$. Now we give conditions under which the $H$-module structure on $\opH^\bullet(B,V)$ may also be described in terms of an $A$-injective resolution of $V$.

Recall the bimodule bar resolution $\Bbul(A,A) = A \otimes A_+^{\otimes \bullet} \otimes A$, with differential $d_n: \upB_n(A,A) \rightarrow \upB_{n-1}(A,A)$ defined by $d_n = \sum_{i=0}^n (-1)^i (1^{\otimes i} \otimes m \otimes 1^{\otimes n-i})$, where $m: A \otimes A \rightarrow A$ is the multiplication in $A$. Form the complex $\Qbul(V) = \Hom_A(\Bbul(A,A),V)$, where $\Hom_A(-,V)$ is taken with respect to the left $A$-module structure of $\Bbul(A,A)$. The right $A$-module structure of $\upB_n(A,A)$ induces the structure of a left $A$-module on $Q_n(V)$. Then $\Qbul(V)$ is an $A$-injective resolution of $V$, called the coinduced resolution of $V$ \cite[VI.2]{Barnes:1985}.

As for the bar resolution, $\Bbul(A,A)$ is a complex of right $H$-modules. Define a left $H$-module structure on $\Hom_k(\upB_n(A,A),V)$ by \eqref{eq:Hdiagonalaction}. Since the $A$ and $H$-module structures on $V$ are compatible, this definition makes $\Qbul(V)$ a complex of left $H$-modules. The $A$- and $H$-modules structures on $\Qbul(V)$ are compatible in the sense of Definition \ref{definition:compatiblestructure} because the $A$- and $H$-module structures on $V$ are compatible.

Now suppose that $A$ is right $B$-flat. Then by Lemma \ref{lemma:injective}, the cohomology group $\opH^n(B,V)$ may be computed as either $\opH^n(\Hom_B(\Bbul(B),V))$ or as $\opH^n(\Hom_B(k,\Qbul(V)))$. From Lemma \ref{lemma:stabilizesinvariants} and Theorem \ref{theorem:Hadjointaction}, we conclude the existence of two possibly inequivalent $H$-module structures on $\opH^n(B,V)$, namely, the adjoint action of $H$ on $\opH^n(B,V)$, and the $H$-module structure induced by the $H$-module structure of $\Qbul(V)$. In fact, these two $H$-module structures are equivalent.

\begin{proposition} \label{proposition:equivalentaction}
The two left $H$-module structures on $\opH^\bullet(B,V)$ defined above are equivalent via a natural isomorphism $\opH^n(\Hom_B(\Bbul(B),V)) \stackrel{\sim}{\rightarrow} \opH^n(\Hom_B(k,\Qbul(V)))$.
\end{proposition}

\begin{proof}
One observes that the natural isomorphism
\[
\opH^n(\Hom_B(\Bbul(B),V)) \stackrel{\sim}{\rightarrow} \opH^n(\Hom_B(k,\Qbul(V)))
\]
constructed by Osborne for the proof of \cite[Corollary 3.12]{Osborne:2000} is a homomorphism of $H$-modules.
\end{proof}

\subsection{Actions on spectral sequences} \label{subsection:spectralsequences}

Our goal now is to show that the action of a Hopf algebra $H$ on an augmented algebra $A$ with normal subalgebra $B$ is well-behaved with respect to the Lyndon--Hochschild--Serre (LHS) spectral sequence associated to the pair $(A,B)$. For future reference in Section \ref{subsection:fgborel}, we recall a construction of the LHS spectral sequence.

\begin{theorem}[Lyndon--Hochschild--Serre Spectral Sequence] \label{theorem:LHSspectralsequence}
Let $A$ be an augmented algebra over $k$, and let $B$ be a normal subalgebra of $A$ such that $A$ is right $B$-flat. Let $V$ be a left $A$-module. Then there exists a spectral sequence satisfying
\begin{equation} \label{eq:spectralsequence}
E_2^{i,j} = \opH^i(A//B,\opH^j(B,V)) \Rightarrow \opH^{i+j}(A,V).
\end{equation}
\end{theorem}

\begin{proof}[Summary of the construction]
We follow the construction of \cite[Chapter VI]{Barnes:1985}. Let $P^\bullet = \Bbul(A//B)$ be the left bar resolution of $A//B$, and let $\Qbul = \Qbul(V)$ be the coinduced resolution of $V$. Form the first quadrant double complex $C = C^{i,j} = \Hom_A(P^i,Q_j)$. There exist two canonical filtrations on $C$, the column-wise filtration $F_I^\bullet$, and the row-wise filtration $F_{II}^\bullet$, each of which gives rise to a spectral sequence converging to $\opH^\bullet(\Tot(C))$, the cohomology of the total complex \cite[Theorem 2.16]{McCleary:2001}. The spectral sequence determined by $F_{II}^\bullet$ collapses at the $E_2$-page and converges to $\opH^\bullet(A,V)$, while the $E_2^{i,j}$-term of the spectral sequence determined by $F_I^\bullet$ is as identified in (\ref{eq:spectralsequence}). Thus, the desired spectral sequence is the one determined by the column-wise filtration $F_I^\bullet$ of the complex $C$.
\end{proof}

Now let $H$ be a Hopf algebra, $A$ a right $H$-module algebra, and $B$ a normal subalgebra of $A$ stable under $H$. Then $A//B$ inherits from $A$ the structure of a right $H$-module algebra, making $P^\bullet = \Bbul(A//B)$ a complex of right $H$-modules. Suppose that the $A$- and $H$-module structures on $V$ are compatible. Then, for each $i,j \in \N$, \eqref{eq:Hdiagonalaction} defines a left $H$-module structure on $C^{i,j}$, and this makes $\Tot(C)$ a complex of $H$-modules. Moreover, the filtrations $F_I^\bullet$ and $F_{II}^\bullet$ of $\Tot(C)$ are filtrations by $H$-submodules.

\begin{theorem} \label{theorem:HmoduleLHSspecseq}
Maintain the notations and assumptions of Theorem \ref{theorem:LHSspectralsequence} and of the previous paragraph. Then \eqref{eq:spectralsequence} is a spectral sequence of left $H$-modules. The $H$-module actions on the $E_2$ page and on the abutment are the adjoint actions of $H$ defined in Definition \ref{definition:adjointaction}.
\end{theorem}

\begin{proof}
This follows from the given construction of \eqref{eq:spectralsequence} and from the results in Section \ref{subsection:actionsviainjectives}.
\end{proof}

Suppose $A$ is a bialgebra and $B$ is a normal sub-bialgebra of $A$. Then the LHS spectral sequence \eqref{eq:spectralsequence} admits cup products \cite[VI.3]{Barnes:1985}. Cup products can also be constructed under weaker conditions on $A$ and $B$, though a different construction for \eqref{eq:spectralsequence} is then required. The following theorem will be utilized in Section \ref{subsection:fgborel}.

\begin{theorem} \label{theorem:LHSspecseqofalgebras}
Let $A$ be an augmented algebra over $k$, and $B$ a central subalgebra of $A$. Assume that $A$ is right $B$-free. Then there exists a spectral sequence of algebras with
\begin{equation} \label{eq:centralLHSspecseq}
E_2^{i,j} = \opH^i(A//B,k) \otimes \opH^j(B,k) \Rightarrow \opH^{i+j}(A,k).
\end{equation}
Let $H$ be a cocommutative Hopf algebra, and suppose $A$ is a right $H$-module algebra, and $B$ is an $H$-submodule of $A$. Then \eqref{eq:centralLHSspecseq} is a spectral sequence of left $H$-module algebras.
\end{theorem}

\begin{proof}
See the construction in the proof of \cite[Theorem 9.12]{McCleary:2001}.
\end{proof}

One can show that the spectral sequences \eqref{eq:spectralsequence} and \eqref{eq:centralLHSspecseq} are isomorphic from the $E_2$-page onward; see \cite[VIII.3]{Barnes:1985}. Implicit in \eqref{eq:centralLHSspecseq} is the fact that when $B$ is central in $A$, the action of $A//B$ on $\opH^\bullet(B,k)$ is trivial; see \cite[Lemma 5.2.2]{Ginzburg:1993}.

\section{Cohomology of the first Frobenius--Lusztig kernel} \label{section:firstFLkernel}

Our goal now is to study the cohomology ring $\opH^\bullet(\Uz(G_1),k)$ for the Frobenius--Lusztig kernel $\Uz(G_1)$ of $\Uz$. The first step is to compute the cohomology ring $\opH^\bullet(\uz,k)$ for the small quantum group $\uz$.

\subsection{Cohomology of the small quantum group} \label{subsection:uzgcohomology}

The strategy for computing the cohomology ring $\opH^\bullet(\uz,k)$ when $p:= \chr(k) > 0$ is largely analogous to the strategy of Ginzburg and Kumar for the case $k = \C$. We summarize the main points of the computation for later reference in Section \ref{subsection:restrictionmaps}. For more details on the case $p > 0$, see \cite{Drupieski:2009}.

\begin{theorem} \label{theorem:inductionspectralsequence}
There exists a first quadrant spectral sequence of $G$-modules satisfying
\begin{equation} \label{eq:inductionspectralsequence}
E_2^{i,j} = R^i \ind_B^G \opH^j(\uzb, k) \Rightarrow \opH^{i+j}(\uz,k).
\end{equation}
\end{theorem}

\begin{proof}
Define functors $\calF_1$ and $\calF_2$ from the category of integrable $\Uzb$-modules to the category of rational $G$-modules by
\begin{equation} \label{eq:equivalentfunctors}
\begin{split}
\calF_1(-) &= (-)^{\uz} \circ \opH^0(\Uz/\Uzb,-), \text{ and} \\
\calF_2(-) &= \ind_B^G(-) \circ (-)^{\uzb}.
\end{split}
\end{equation}
Here we have identified the category of rational $G$-modules with the category of locally finite $\Dist(G)$-modules \cite{Sullivan:1978b}, and the category of rational $B$-modules with the category of integrable $\Dist(B)$-modules \cite[Theorem 9.4]{Cline:1980}. The induction functors $\opH^0(\Uz/\Uzb,-)$ and $\ind_B^G(-)$ are left exact and take injective modules to injective modules. The fixed-point functor $(-)^{\uz}$ is right adjoint to the (exact) forgetful functor $(-)^{[1]}$ from the category of of rational $G$-modules to integrable $\Uz$-modules, so maps injective modules to injective modules. Similarly, the fixed-point functor $(-)^{\uzb}$ maps injective modules to injective modules.

In the definition of $\calF_1$, $(-)^{\uz}$ is considered as a functor from the category of integrable $\Uz$-modules to the category of rational $G$-modules. The right derived functors $R^i(-^{\uz})$ are defined in terms of injective resolutions by integrable $\Uz$-modules. Since the induction functor $\opH^0(\Uz/\uz,-)$ is exact by \cite[Corollary 2.3]{Andersen:1992a}, it follows from a standard argument (see, e.g., \cite[Proposition 2.1]{Cline:1977}) that injective integrable modules $\Uz$-modules are injective for $\uz$. Then we may identify identify $R^i(-^{\uz})$ with $\opH^i(\uz,-)$. Similarly, we may identify the right derived functors $R^i(-^{\uzb})$ with $\opH^i(\uzb,-)$.

The functors $\calF_1$ and $\calF_2$ are both right adjoint to the functor $(-)^{[1]}|_{\Uzb}$, hence are naturally isomorphic. Then by \cite[I.4.1]{Jantzen:1996}, there exist spectral sequences
\begin{align*}
E_2^{i,j} &= \opH^i(\uz,\opH^j(\Uz/\Uzb,k)) \Rightarrow (R^{i+j} \calF_1)(k), \text{ and} \\
E_2^{i,j} &= R^i \ind_B^G \opH^j(\uzb,k) \Rightarrow (R^{i+j} \calF_2)(k),
\end{align*}
converging to the same abutment. By \cite[Theorem 5.5]{Ryom-Hansen:2003}, $\opH^i(\Uz/\Uzb,k) = 0$ for all $i>0$. Then the first spectral sequence collapses at the $E_2$ page, giving $(R^\bullet \calF_1)(k) \cong (R^\bullet \calF_2)(k) \cong \opH^\bullet(\uz,k)$. So the second spectral sequence is the spectral sequence of the theorem.
\end{proof}

\begin{remarks} \ 
\begin{enumerate}
\item In the case $p=0$, Andersen, Polo and Wen \cite{Andersen:1992,Andersen:1991} prove by different methods that every (finite-dimensional) injective $\Uz$-module restricts to an injective $\uz$-module. 

\item It follows from Proposition \ref{proposition:equivalentaction} that the $G$-module structure on $\opH^\bullet(\uz,k)$ in \eqref{eq:inductionspectralsequence} is equivalent to the adjoint action of $\Dist(G) \cong \Uz//\uz$ on $\opH^\bullet(\uz,k)$.
\end{enumerate}
\end{remarks}

By Lemma \ref{lemma:cohomologyHmodulealgebra} and Corollary \ref{corollary:uzbstable}, the ring $\opH^\bullet(\uzb,k)$ is a $\Uzb$-module algebra.

\begin{theorem} \label{theorem:uzbalgebrastructure}
Suppose $\ell > h$. Let $S^\bullet(\fraku^*)$ be the symmetric algebra on $\fraku^*$, with $B$-module (equivalently, $\Dist(B)$-module) structure induced by the adjoint action of $B$ on $\fraku$. Then $\opH^{\textup{odd}}(\uzb,k) = 0$, and there exists an isomorphism of $\Uzb$-module algebras $\opH^{2\bullet}(\uzb,k) \cong S^\bullet(\fraku^*)^{[1]}$.
\end{theorem}

\begin{proof}
The argument in \cite[\S 2.5]{Ginzburg:1993} establishing $\opH^{2\bullet}(\uzb,k) \cong S^\bullet(\fraku^*)^{[1]}$ as $\Uzo$-module algebras in the case $p = 0$ applies equally well if $p > 0$. On the other hand, the argument in \cite[\S 2.6]{Ginzburg:1993} for the $\Uzb$-action on $\opH^\bullet(\uzb,k)$ does not generalize to the case $p > 0$, because then $\Dist(U) \cong \Uzu//\uzu$ is not generated by the images of the $\ell$-th divided powers in $\Uz^-$. Still, one can show that $\opH^2(\uzb,k) \cong (\fraku^*)^{[1]}$ as a $\Uzb$-module; see \cite[\S 4.2]{Drupieski:2009}. Then Lemma \ref{lemma:cohomologyHmodulealgebra} and the fact that $\opH^\bullet(\uzb,k)$ is generated in degree two imply that the $\Uzb$-module isomorphism $\opH^2(\uzb,k) \cong (\fraku^*)^{[1]}$ extends to an isomorphism of $\Uzb$-module algebras $\opH^{2\bullet}(\uzb,k) \cong S^\bullet(\fraku^*)^{[1]}$.
\end{proof}

Recall that $p = \chr(k)$ is said to be good for the root system $\Phi$ if $p > 2$ (resp. $p > 3$, $p > 5$) when $\Phi$ has a component not of type $A$ (resp.\ has a component of exceptional type, resp.\ of type $E_8$). We now finish the computation of $\opH^\bullet(\uz,k)$ when $\ell > h$ and when $p = \chr(k)$ is good for $\Phi$.

\begin{theorem} \label{theorem:uzgcohomology}
Suppose $\ell > h$, and that $k$ is algebraically closed of characteristic good for $\Phi$. Then $\opH^{\textup{odd}}(\uz,k) = 0$, and there exist $G$-module algebra isomorphisms $\opH^{2\bullet}(\uz,k) \cong \ind_B^G S^\bullet(\fraku^*) \cong k[\calN]$, where $k[\calN]$ is the coordinate ring of the variety $\calN$ of nilpotent elements in $\g = \Lie(G)$. In particular, $\opH^\bullet(\uzg,k)$ is finitely generated as a ring.
\end{theorem}

\begin{proof}
Since $p$ is good for $\Phi$, we have $R^i \ind_B^G S^\bullet(\fraku^*) = 0$ for all $i > 0$ by \cite[Theorem 2]{Kumar:1999}. Then the spectral sequence \eqref{eq:inductionspectralsequence} collapses at the $E_2$-page, yielding $\opH^{\textup{odd}}(\uz,k) = 0$, and the $G$-module isomorphism $\opH^{2\bullet}(\uz,k) \cong \ind_B^G S^\bullet(\fraku^*)$. The isomorphism $\opH^{2\bullet}(\uz,k) \cong \ind_B^G S^\bullet(\fraku^*)$ is an isomorphism of algebras by the argument in \cite[Remark 3.2]{Andersen:1984}. For $p$ good, the map $\rho: S^\bullet(\g^*) \rightarrow \ind_B^G S^\bullet(\fraku^*)$ induced by Frobenius reciprocity from the restriction map $S^\bullet(\g^*) \rightarrow S^\bullet(\fraku^*)$ induces a $G$-module algebra isomorphism $k[\calN] \stackrel{\sim}{\rightarrow} \ind_B^G S^\bullet(\fraku^*)$; see the argument in \cite[\S 3.5]{Nakano:1996a} (cf.\ also \cite[\S 6.20]{Humphreys:1995}).
\end{proof}

The $G$-module structure of $\opH^\bullet(\uz,k)$ can also be determined for most values of $\ell$ smaller than $h$; see \cite{Bendel:2011,Drupieski:2009}. The last theorem of this section was obtained in the special case $\chr(k) = 0$ by Mastnak, Pevtsova, Schauenburg and Witherspoon \cite{Mastnak:2010} as a corollary to their study of the cohomology of finite-dimensional pointed Hopf algebras. Their techniques apply equally well to the study of $\opH^\bullet(\uz,k)$ when $\chr(k) > 0$, because the group $G(\uz)$ of grouplike elements in $\uz$ is isomorphic to $(\Z/\ell\Z)^n$, hence semisimple over $k$.

\begin{theorem} \label{theorem:mastnak} \textup{\cite[Corollary 6.5]{Mastnak:2010}}
The cohomology ring $\opH^\bullet(\uz,k)$ is finitely generated. For any finite-dimensional $\uz$-module $M$, $\opH^\bullet(\uz,M)$ is finitely-generated as a module over $\opH^\bullet(\uz,k)$.
\end{theorem}

\subsection{Restriction maps} \label{subsection:restrictionmaps}

Let $\Phi_J \subseteq \Phi$ be an indecomposable root subsystem of $\Phi$ corresponding to a subset of simple roots $J \subseteq \Pi$. Set $\Phi' = \Phi_J$, and let $\uz(\g'),\uz(\frakb'),\uz(\fraku')$ be the small quantum groups defined in terms of $\Phi'$. Then the inclusion of root systems $\Phi' \subseteq \Phi$ induces injective algebra homomorphisms $\uz(\g') \rightarrow \uzg$, $\uz(\frakb') \rightarrow \uzb$, and $\uz(\fraku') \rightarrow \uzu$. For example, the map $\uz(\g') \rightarrow \uzg$ maps $E_\alpha \mapsto E_\alpha$, $F_\alpha \mapsto F_\alpha$, and $K_\alpha \mapsto K_\alpha$ for all $\alpha \in J$. Now suppose $\ell > h$, where $h$ is the Coxeter number of $\Phi$. Then also $\ell > h'$, the Coxeter number of $\Phi'$, so Theorem \ref{theorem:uzbalgebrastructure} applies to both $\opH^\bullet(\uzb,k)$ and $\opH^\bullet(\uz(\frakb'),k)$, and Theorem \ref{theorem:uzgcohomology} applies to both $\opH^\bullet(\uzg,k)$ and $\opH^\bullet(\uz(\g'),k)$.

\begin{lemma} \textup{\cite[Proposition 5.7]{Drupieski:2009}}  \label{lemma:uzbrestriction}
Suppose $\ell > h$. Then under the identifications of Theorem \ref{theorem:uzbalgebrastructure}, the restriction map $\opH^{2\bullet}(\uzb,k) \rightarrow \opH^{2\bullet}(u_\zeta(\frakb'),k)$ is simply the restriction of functions from $\fraku$ to $\fraku'$.
\end{lemma}

Let $G'$ be the simple, simply-connected algebraic group over $k$ with $\g' = \Lie(G')$, and let $B' \subset G'$ be the Borel subgroup of $G'$ with $\frakb' = \Lie(B')$. Write $\calN'$ for the variety of nilpotent elements in $\g'$.

\begin{lemma} \label{lemma:uzgrestriction}
Suppose $\ell > h$ and that $k$ is algebraically closed of characteristic good for both $\g$ and $\g'$. Then under the identifications $\opH^{2\bullet}(\uzg,k) \cong k[\calN]$ and $\opH^{2\bullet}(\uz(\g'),k) \cong k[\calN']$ of Theorem \ref{theorem:uzgcohomology}, the restriction homomorphism $\opH^{2\bullet}(\uzg,k) \rightarrow \opH^{2\bullet}(\uz(\g'),k)$ is just the restriction of functions $k[\calN] \rightarrow k[\calN']$.
\end{lemma}

\begin{proof}
Let $\calF_1',\calF_2'$ be the functors from the category of integrable $\Uz(B')$-modules to the category of rational $G'$-modules defined by substituting the symbols $\g',\frakb',B',G'$ for the symbols $\g,\frakb,B,G$ in \eqref{eq:equivalentfunctors}. Any integrable $\Uzb$-module is by restriction an integrable $\Uz(B')$-module. Now define natural transformations $\eta_1: \calF_1 \rightarrow \calF_1'$ and $\eta_2: \calF_2 \rightarrow \calF_2'$ as follows:
\begin{itemize}
\item The evaluation map $\varepsilon: \opH^0(\Uzg/\Uzb,M) \rightarrow M$ is a homomorphism of $U_\zeta(\frakb')$-modules. By Frobenius reciprocity, there exists a map
\[ \opH^0(\Uzg/\Uzb,M) \rightarrow \opH^0(\Uz(\g')/\Uz(\frakb'),M).
\]
Call this map $\ind(\varepsilon)$. Define $\eta_1$ to be the restriction of $\ind(\varepsilon)$ to the subspace $\calF_1(M) \subset \opH^0(\Uzg/\Uzb,M)$. Then $\eta_1$ has image in $\calF_1'(M)$.

\item The evaluation map $ \varepsilon: \calF_2(M) = \ind_B^G(M^{\uzb}) \rightarrow M^{\uzb} \subset M^{u_\zeta(\frakb')}$ is a $B'$-module homomorphism. By Frobenius reciprocity, there exists a corresponding $G'$-module homomorphism $\eta_2: \calF_2(M) \rightarrow \calF_2'(M)$.
\end{itemize}
Let $\theta: \calF_1 \stackrel{\sim}{\rightarrow} \calF_2$ be the natural equivalence arising from the fact that $\calF_1$ and $\calF_2$ are both right adjoint to the functor $(-)^{[1]}|_{\Uzb}$. Similarly, let $\theta': \calF_1' \stackrel{\sim}{\rightarrow} \calF_2'$ be the natural equivalence for $\calF_1'$ and $\calF_2'$. Then $\eta_1$ and $\eta_2$ commute with $\theta$ and $\theta'$, i.e., $\theta' \circ \eta_1 = \eta_2 \circ \theta$.

The natural transformations $\eta_1: \calF_1 \rightarrow \calF_1'$ and $\eta_2: \calF_2 \rightarrow \calF_2'$ induce morphisms of the higher derived functors, $\eta_1^\bullet: R^\bullet \calF_1 \rightarrow R^\bullet \calF_1'$ and $\eta_2^\bullet: R^\bullet \calF_2 \rightarrow R^\bullet \calF_2'$. One checks that under the identifications $(R^\bullet \calF_1)(k) \cong \opH^\bullet(\uzg,k)$ and $(R^\bullet \calF_1')(k) \cong \opH^\bullet(\uz(\g'),k)$, the morphism $\eta_1^\bullet: \opH^\bullet(\uzg,k) \rightarrow \opH^\bullet(\uz(\g'),k)$ is the restriction homomorphism, and under the identifications $(R^\bullet \calF_2)(k) \cong \ind_B^G \opH^\bullet(\uzb,k)$ and $(R^\bullet \calF_2')(k) \cong \ind_{B'}^{G'} \opH^\bullet(\uz(\frakb'),k)$, the homomorphism $\eta_2^\bullet: \ind_B^G \opH^\bullet(\uzb,k) \rightarrow \ind_{B'}^{G'} \opH^\bullet(\uz(\frakb'),k)$ is the $G'$-module homomorphism induced by Frobenius reciprocity from the restriction map $\opH^\bullet(\uzb,k) \rightarrow \opH^\bullet(\uz(\frakb'),k)$.

By Lemma \ref{lemma:uzbrestriction}, the restriction map $\opH^{2\bullet}(\uzb,k) \rightarrow \opH^{2\bullet}(\uz(\frakb'),k)$ identifies with the restriction of functions $S^\bullet(\fraku^*) \rightarrow S^\bullet(\fraku'^*)$. Then by Frobenius reciprocity, the maps
\[
f: k[\calN] \stackrel{\sim}{\rightarrow} \ind_B^G \opH^\bullet(\uzb,k) \stackrel{\eta_2^\bullet}{\rightarrow} \ind_{B'}^{G'} \opH^\bullet(\uz(\frakb'),k)
\]
and
\[
g: k[\calN] \stackrel{\res}{\rightarrow} k[\calN'] \stackrel{\sim}{\rightarrow} \ind_{B'}^{G'} \opH^\bullet(\uz(\frakb'),k)
\]
must be the same, because they are both $G'$-module homomorphisms whose compositions with the evaluation map $\ind_{B'}^{G'} \opH^\bullet(\uz(\frakb'),k) \rightarrow \opH^\bullet(\uz(\frakb'),k) \cong S^\bullet(\fraku'^*)$ are the restriction of functions from $\calN$ to $\fraku'$.
\end{proof}

\subsection{Cohomology of the first Frobenius--Lusztig kernel}

Friedlander and Parshall \cite{Friedlander:1986} were able to prove the finite-generation of the cohomology ring $\opH^\bullet(G_2,k)$ for the second Frobenius kernel of $G$ by studying the Lyndon--Hochschild--Serre spectral sequence
\[
E_2^{i,j} = \opH^i(G_2/G_1,\opH^j(G_1,k)) \cong \opH^i(G_1,\opH^j(G_1,k)^{(-1)}) \Rightarrow \opH^{i+j}(G_2,k).
\]
We now imitate their approach to study the cohomology ring for the first Frobenius--Lusztig kernel $\Uz(G_1)$ of $\Uz$. Throughout this section, assume $k$ to be algebraically closed, with $p = \chr(k)$ odd and very good for $G$ (i.e., $p$ is good for $G$, and $p \nmid n+1$ if $\Phi$ has type $A_n$). Also assume $\ell > h$, so that the algebra isomorphism $\opH^{2\bullet}(\uzg,k) \cong \ind_B^G S^\bullet(\fraku^*)$ of Theorem \ref{theorem:uzgcohomology} holds.

\begin{lemma} \label{lemma:kernelfreeoveranother}
Fix integers $0 \leq r \leq s \leq \infty$. If $s = \infty$, set $\Uz(G_s) = \Uz$. Then $\Uz(G_s)$ is free (in particular, flat) for both the left and right regular actions of $\Uzgr$ on $\Uz(G_s)$.
\end{lemma}

\begin{proof}
To prove that $\Uz(G_s)$ is free as a left $\Uzgr$-module, it suffices by \cite[Corollary 1.7]{Koppinen:1983} to show that the left regular representation for $\Uzgr$ lifts to $\Uz(G_s)$. For this we follow the strategy in the proof of \cite[Theorem 5.1(i)]{Koppinen:1983}. Then the freeness of $\Uz(G_s)$ as a right $\Uzgr$-module follows by applying the antipode $S$ of $\Uz$.

By \eqref{eq:SteinbergZiso} and \cite[Corollary 3.5]{Drupieski:2011a}, the $p^r\ell$-th Steinberg module $\Stprl$ is free as a left $\Uzur$-module. Since $\Uzur$ is a subalgebra of the Hopf algebra $\Uzgr$, there exists by restriction a $\Uzur$-module isomorphism $\Stprl \cong k \otimes \Stprl$. Write $\coind(-) = \Uzgr \otimes_{\Uzur}(-)$. We claim that $\Uzgr \cong \coind(k \otimes \Stprl) \cong \coind(k) \otimes \Stprl$ as a left $\Uzgr$-module. If $\Uzur$ were a Hopf subalgebra of $\Uzgr$, then this would follow from the usual tensor identity for the tensor induction functor. Nevertheless, the usual maps giving the inverse isomorphisms of the tensor identity yield in this case an isomorphism $\coind(k \otimes \Stprl) \cong \coind(k) \otimes \Stprl$. Specifically, the following linear maps are well-defined $\Uzgr$-homomorphisms:
\begin{align*}
\varphi &: \coind(k \otimes \Stprl) \rightarrow \coind(k) \otimes \Stprl & h \otimes (v \otimes w) \mapsto \sum (h_{(1)} \otimes v) \otimes h_{(2)}w \\
\psi &: \coind(k) \otimes \Stprl \rightarrow \coind(k \otimes \Stprl) & (h \otimes v) \otimes w \mapsto \sum h_{(1)} \otimes (v \otimes S(h_{(2)}) w)
\end{align*}
Here $h \in \Uzgr$, $v \in k$, and $w \in \Stprl$. The well-definedness here of $\varphi$ and $\psi$ is dependent on the fact that the first factor in $k \otimes \Stprl$ is the trivial module.

So now $\Uzgr \cong \coind(k) \otimes \Stprl$ as a left $\Uzgr$-module. Let $L_1,\ldots,L_r$ be the $\Uzgr$-composition factors for $\coind(k)$. Then the left regular representation of $\Uzgr$ admits a filtration with quotients $L_1 \otimes \Stprl, \ldots, L_r \otimes \Stprl$. Since $\St_\ell$ is projective for $\Uzgr$ by Corollary \ref{corollary:Steinbergprojective}, so is each $L_i \otimes \Stprl$, hence there exists an isomorphism of left $\Uzgr$-modules $\Uzgr \cong (L_1 \oplus \cdots \oplus L_r) \otimes \Stprl$. Since each $L_i$ can be lifted to a simple $\Uz$-module by Theorem \ref{theorem:simplerestrictsimple}, $L_1 \oplus \cdots \oplus L_r \cong V|_{\Uzgr}$ for some completely reducible $\Uz$-module $V$. Then as a left $\Uzgr$-module by , $\Uzgr \cong (V \otimes \Stprl)|_{\Uzgr}$, the restriction to $\Uzgr$ of a $\Uz$-module. In particular, the left regular representation of $\Uzgr$ lifts to $\Uz(G_s)$.
\end{proof}

By Lemma \ref{lemma:kernelfreeoveranother} and Theorem \ref{theorem:HmoduleLHSspecseq}, there exists a spectral sequence of $\Uz$-modules satisfying
\begin{equation} \label{eq:secondkernelspecseq}
E_2^{i,j}(\g) = \opH^i(\Uz(G_1)//\uzg,\opH^j(\uzg,k)) \Rightarrow \opH^{i+j}(\Uz(G_1),k).
\end{equation}
Applying the isomorphism $U_\zeta(G_1)//\uzg \cong \Dist(G_1)$ and the results of Section \ref{subsection:uzgcohomology}, we rewrite \eqref{eq:secondkernelspecseq} as
\begin{equation} \label{eq:G1specseq}
E_2^{i,j}(\g) = \opH^i(G_1,\ind_B^G S^{j/2}(\fraku^*)) \Rightarrow \opH^{i+j}(\Uz(G_1),k).
\end{equation}
In particular, $E_2^{i,j}(\g) = 0$ unless $j$ is even.

Let $\nu$ denote the highest root in $\Phi$. If $\Phi$ has only one root length, then $\nu$ is the minimal element among the non-zero dominant weights lying in the root lattice.

\begin{lemma} \label{lemma:fp15} \textup{\cite[Lemma 1.5]{Friedlander:1986}}
Suppose $\Phi$ has rank $n$. Let $w \in W$ be such that $-w \cdot 0 = \rho - w \rho \geq s \nu$ for some positive integer $s$. Then $\ell(w) \geq n+s-1$.
\end{lemma}

\begin{proposition} \label{proposition:vanishingrange} In the spectral sequence \eqref{eq:G1specseq}, suppose $E_2^{i,j}(\g) \neq 0$ with $i+j=2p+1$. Write $j=2p-2s$ for some $0 \leq s \leq p$.
\begin{enumerate}
\item If $\Phi$ is of type $A_n$, then $n-2 \leq s \leq n$.
\item If $\Phi$ is of type $D_n$, then $n-2 \leq s \leq 2(n-1)$.
\end{enumerate}
\end{proposition}

\begin{proof}
The proof here follows exactly the strategy of \cite[Proposition 1.6]{Friedlander:1986}. We provide the details in order to show that the argument extends to good characteristics. (The original result is proven under the assumption $p>h$.) Set $n = \rank(\Phi)$, and write $\Pi = \set{\alpha_1,\ldots,\alpha_n}$, with the simple roots ordered as in \cite{Humphreys:1978}.

Since $p$ is good for $G$, the rational $G$-module $\ind_B^G S^{j/2}(\fraku^*)$ admits a good filtration \cite[II.12.12]{Jantzen:2003}. The non-vanishing of $E_2^{i,j}(\g)$ then implies that there exists a weight $\mu \in X^+$ such that $\mu$ is a weight of $\ind_B^G S^{j/2}(\fraku^*)$, and such that $\opH^i(G_1,\opH^0(\mu)) \neq 0$. If $j>0$, then $\mu = w \cdot 0 + p\lambda \neq 0$ for some $\lambda \in X^+$ and some $w \in W$ with $\ell(w) \leq i$. Indeed, the proof of \cite[Corollary 5.5]{Andersen:1984} (which establishes the given form for $\mu$ in the classical $p> h$ case) remains valid for $p$ good if we apply the stronger form of \cite[Proposition 5.4]{Andersen:1984} proved in \cite[Theorem 2]{Kumar:1999}.

It follows from the isomorphism $k[\calN] \cong \ind_B^G S^\bullet(\fraku^*)$ that any weight $\mu$ of $\ind_B^G S^{p-s}(\fraku^*)$ must satisfy $\mu \leq (p-s)\nu$, because $k[\calN]$ is a quotient of $S^\bullet(\g^*)$, and weights of $S^{p-s}(\g^*)$ must be less than or equal to $(p-s)\nu$. Since $p$ is very good for $G$ (by assumption), we have $p \nmid [X:\Z\Phi]$, hence $\lambda = (\mu-w\cdot 0)/p \in X^+$ must belong to the root lattice. This implies that $\lambda \geq \nu$ by the comment immediately preceding Lemma \ref{lemma:fp15}. Now we get $-w \cdot 0 = p\lambda - \mu \geq s\nu + p(\lambda-\nu) \geq s\nu$, so Lemma \ref{lemma:fp15} implies that $i \geq \ell(w) \geq n+s-1$. The inequality $-w \cdot 0 \geq s\nu$ also implies $n \geq s$ (resp.\ $2(n-1) \geq s$) if $\Phi$ has type $A_n$ (resp.\ type $D_n$), because there are only $n$ (resp.\ $2(n-1)$) roots greater than or equal to $\alpha_1$ in $\Phi$. Since $i = 2s+1$, this proves the proposition.
\end{proof}

We can now prove the main theorem of this section.

\begin{theorem} Assume $k$ to be algebraically closed with $p=\chr(k)$ odd and very good for $G$. Assume also that $\ell$ is odd, $\ell > h$, and that one of the following conditions is satisfied:
\begin{enumerate}
\item $\Phi$ is either of type $A_n$ or of type $D_n$, and $n > p+2$,
\item $\Phi$ is of type $A_n$ and $\ell \geq n+4 = h+3$, or
\item $\Phi$ is of type $D_n$ and $\ell \geq 4n = 2h+4$.
\end{enumerate}
Then for any finite-dimensional $U_\zeta(G_1)$-module $M$, $\opH^\bullet(U_\zeta(G_1),M)$ is a finite module for the Noetherian algebra $\opH^\bullet(U_\zeta(G_1),k)$.
\end{theorem}

\begin{proof}
By Theorem \ref{theorem:HmoduleLHSspecseq}, there exists a spectral sequence satisfying
\begin{equation} \label{eq:secondkernelspecseqM}
E_2^{i,j}(M) = \opH^i(\Uz(G_1)//\uzg,\opH^j(\uzg,M)) \Rightarrow \opH^{i+j}(\Uz(G_1),M).
\end{equation}
Identifying $\Uz(G_1)//\uzg$ with $\Dist(G_1)$, we rewrite \eqref{eq:secondkernelspecseqM} as
\begin{equation}
E_2^{i,j}(M) = \opH^i(G_1,\opH^j(\uzg,M)) \Rightarrow \opH^{i+j}(\Uz(G_1),M).
\end{equation}
Moreover, since $\uz$ is a Hopf subalgebra of $\Uz(G_1)$, the spectral sequence $E_r(M)$ is a module over the spectral sequence $E_r(k)$.

Applying Theorem \ref{theorem:mastnak}, we obtain the following situation: $\opH^\bullet(\uzg,k)$ is a Noetherian $k$-algebra on which $G_1$ acts rationally by $k$-algebra automorphisms, $\opH^\bullet(\uzg,M)$ is a rational $G_1$-module on which $\opH^\bullet(\uzg,k)$ acts compatibly, and $\opH^\bullet(\uzg,M)$ is a finite module for $\opH^\bullet(\uzg,k)$. Then by \cite[Lemma 3.3, Theorem 3.5]{Kallen:2007}, $E_2^{\bullet,\bullet}(M) = \opH^\bullet(G_1,\opH^\bullet(\uzg,M))$ is a finite module for the Noetherian algebra $E_2^{\bullet,\bullet}(\g) := \opH^\bullet(G_1,\opH^\bullet(\uzg,k))$. To prove the assertion of the theorem, it now suffices by a standard argument (cf.\ \cite[Lemmas 7.4.4, 7.4.5]{Evens:1991}) to show that $E_2^{\bullet,\bullet}(\g)$ is finitely-generated over a Noetherian subalgebra of permanent cycles.

Define $S \subset \opH^{2p}(\uzg,k)$ to be the vector subspace spanned by all $p$-th powers of elements of $\opH^2(\uzg,k)$, and let $R \subset \opH^\bullet(\uzg,k)$ be the subalgebra generated by $S$. Evidently, $R \subset \opH^0(G_1,\opH^\bullet(\uzg,k))$, because $G_1$ acts trivially on all $p$-th powers of elements in $\opH^\bullet(\uzg,k)$. Also, $\opH^\bullet(\uzg,k)$ is finitely-generated over $R$. Applying \cite[Lemma 3.3]{Kallen:2007} again, we conclude that $E_2^{\bullet,\bullet}(\g)$ is finitely-generated over the subalgebra $\opH^\bullet(G_1,R) = \opH^\bullet(G_1,k) \otimes R$. We claim that $\opH^\bullet(G_1,R)$ consists of permanent cycles. Since the differential of \eqref{eq:secondkernelspecseq} is an algebra derivation, it suffices to show that the subspace $S \subset E_2^{0,2p}(\g)$ consists of permanent cycles.

Denote the differential $E_s^{i,j}(\g) \rightarrow E_s^{i+s,j-s+1}(\g)$ of \eqref{eq:secondkernelspecseq} by $d_s^{i,j}(\g)$. To prove the claim for $S$, it suffices to show that $d_{2s+1}^{0,2p}(\g)(S)=0$ for all $1 \leq s \leq p$. (We have used the fact that $E_2^{i,j}(\g)=0$ unless $j$ is even.) Suppose $\Phi$ is of type $A_n$ or $D_n$. According to Proposition \ref{proposition:vanishingrange}, if $d_{2s+1}^{0,2p}(\g) \neq 0$, then $s \geq n-2$. If $n > p+2$, then $n-2 > p$, so $d_{2s+1}^{0,2p}(\g) \equiv 0$ for all $1 \leq s \leq p$. This proves the claim when condition (1) is satisfied.

By assumption, $\rank(\Phi) = n$. For each $m \geq n$, let $\Phi_m$ be the rank $m$ indecomposable root system of the same Lie type as $\Phi$, and let $\g_m$ be the corresponding simple Lie algebra over $k$. Then the inclusion of root systems $\Phi \subseteq \Phi_m$ induces an inclusion of algebras $\Uzg \subset U_\zeta(\g_m)$, hence a morphism of spectral sequences $f_s^{\bullet,\bullet}:E_s^{\bullet,\bullet}(\g_m) \rightarrow E_s^{\bullet,\bullet}(\g)$, such that the map
\[
f_2^{0,\bullet}:E_2^{0,\bullet}(\g_m) \rightarrow E_2^{0,\bullet}(\g)
\]
is induced by the restriction map $\opH^\bullet(u_\zeta(\g_m),k) \rightarrow \opH^\bullet(\uzg,k)$ studied in Section \ref{subsection:restrictionmaps}. If $\ell$ is at least the Coxeter number of $\Phi_m$, so that Theorem \ref{theorem:uzgcohomology} holds for $\uz(\g_m)$ as well as for $\uzg$, then we can apply Lemma \ref{lemma:uzgrestriction} to conclude that $S \subseteq \im (f_2^{0,2p})$.

Now suppose condition (2) is satisfied, so that $\Phi$ has type $A_n$ and $\ell \geq n+4$. Then $\Phi \subset \Phi_{n+3}$, and for each $1 \leq s \leq p$, we have the following commutative diagram (where $\g' = \g_{n+3}$):
\begin{equation} \label{eq:commutativediagram}
\xymatrix@C+5ex{
E_{2s+1}^{0,2p}(\g') \ar[r]_{d_{2s+1}^{0,2p}(\g')} \ar[d]_{f_{2s+1}^{0,2p}} & E_{2s+1}^{2s+1,2p-2s}(\g') \ar[d]^{f_{2s+1}^{2s+1,2p-2s}} \\
E_{2s+1}^{0,2p}(\g) \ar[r]_{d_{2s+1}^{0,2p}(\g)} & E_{2s+1}^{2s+1,2p-2s}(\g)
}
\end{equation}
According to Proposition \ref{proposition:vanishingrange}, $d_{2s+1}^{0,2p}(\g) \equiv 0$ if $1 \leq s \leq (n-3)$ or if $(n + 1) \leq s \leq p$, and $d_{2s+1}^{0,2p}(\g') \equiv 0$ if $1 \leq s \leq n$ (because $n = \rank(\Phi_{n+3})-3$). Since $\ell$ is at least the Coxeter number of $\Phi_{n+3}$, we have $S \subseteq \im(f_2^{0,2p})$. It follows then from the commutativity of \eqref{eq:commutativediagram} that $d_{2s+1}^{0,2p}(\g)(S) = 0$ for $1 \leq s \leq n$, hence that $d_{2s+1}^{0,2p}(\g)(S) = 0$ for all $1 \leq s \leq p$. This proves that the set $S$ consists of permanent cycles whenever condition (2) is satisfied. The proof that $S$ consists of permanent cycles whenever condition (3) is satisfied is similar, and the details are left to the reader. (Embed $\Phi$ in $\Phi_{2n+1}$, which has Coxeter number $2(2n+1)-2 = 4n$. Then argue as for type $A$, using part (b) of Proposition \ref{proposition:vanishingrange}.) 
\end{proof}

\section{Higher Frobenius--Lusztig kernels} \label{section:higherFLkernels}

In this section we study the cohomology rings $\opH^\bullet(\Uzbr,k)$ and $\opH^\bullet(\Uzur,k)$ for the subalgebras $\Uzbr$ and $\Uzur$ of $\Uzgr$ corresponding to the Borel subgroup $B$ and its unipotent radical $U$. In the classical situation, finite-generation for $\opH^\bullet(U_r,k)$ can be proven via an inductive approach, by considering a filtration on $U_r$ by unipotent subgroups. For quantum groups the approach is similar in spirit, though much more technically complicated. The approach we follow here is inspired by the techniques of Ginzburg and Kumar \cite[\S 2.4]{Ginzburg:1993}.

\subsection{Algebra filtrations} \label{subsection:algebrafiltrations}

For $\bfr,\bfs \in \N^N$ and $\mu \in \Z\Phi$, define the total height of the monomial $F^{\bfr}K_\mu E^{\bfs} \in \Uq$ by $\hght(F^{\bfr}K_\mu E^{\bfs}) = \sum_{i=1}^N (r_i+s_i)\hght(\gamma_i)$. Then define the degree $d$ of $F^{\bfr}K_\mu E^{\bfs}$ by	
\begin{equation} \label{eq:degreemonomial}
d(F^{\bfr}K_\mu E^{\bfs}) = (r_N,r_{N-1},\ldots,r_1,s_1,\ldots,s_N,\hght(M_{\bfr,\bfs,u})) \in \N^{2N+1}.
\end{equation}
View $\Lambda:=\N^{2N+1}$ as a totally ordered semigroup via the reverse lexicographic ordering. Given $\eta \in \N^{2N+1}$, define $\U_{q,\eta}$ to be the linear span in $\Uq$ of all monomials $F^{\bfr}K_\mu E^{\bfs}$ with $d(F^{\bfr}K_\mu E^{\bfs}) \leq \eta$. Then the collection of subspaces $\U_{q,\eta}$ for $\eta \in \Lambda$ forms a multiplicative $\Lambda$-filtration of $\Uq$ \cite[\S 10.1]{De-Concini:1993}. This filtration induces multiplicative filtrations on $\Uz$ and on $\Uzgr$.

Fix $r \in \N$. We transform the $\Lambda$-filtration on $\Uzgr$ into an $\N$-filtration as follows. Set $\theta = 2p^r\ell$. Given $\bfr,\bfs \in \N^N$ and $u \in \Uzo$, set $M_{\bfr,\bfs,u} = F^{(\bfr)}uE^{(\bfs)}$. Now define $\deg(M_{\bfr,\bfs,u}) \in \N$ by
\[
\deg(M_{\bfr,\bfs,u}) = r_N + \theta r_{N-1} + \theta^2 r_{N-2} + \cdots + \theta^{N-1}r_1 + \theta^N s_1 + \cdots + \theta^{2N-1}s_N + \theta^{2N} \hght(M_{\bfr,\bfs,u}).
\]
Then $\deg(M_{\bfr,\bfs,u}) \leq \deg(M_{\bfr',\bfs',u'})$ if and only if $d(M_{\bfr,\bfs,u}) \leq d(M_{\bfr',\bfs',u'})$.

\begin{lemma} \label{lemma:Nfiltration}
For $n \in \N$, define $\Uzgr_n$ to be the subspace of $\Uzgr$ spanned by all monomials $M_{\bfr,\bfs,u} \in \Uzgr$ with $\deg(M_{\bfr,\bfs,u}) \leq n$. Then the subspaces $\Uzgr_n$ for $n \in \N$ define a multiplicative $\N$-filtration on $\Uzgr$. The associated graded algebras arising from the $\Lambda$- and $\N$-filtrations on $\Uzgr$ are canonically isomorphic as non-graded algebras.
\end{lemma}

By restriction, we obtain an $\N$-filtration on $\Uzur$. Let $\Phi^+ = \set{\gamma_1,\ldots,\gamma_N}$ be an enumeration of $\Phi^+$ as in Section \ref{subsection:braidgroupauto}. Given $\alpha,\beta \in \Phi^+$, write $\alpha \prec \beta$ if $\alpha = \gamma_i$, $\beta = \gamma_j$, and $i<j$. Now define $\scrA$ to be the twisted polynomial algebra with generators
\begin{equation} \label{eq:scrAgenerators}
\set{X_\alpha,X_{p^i \ell \alpha}:\alpha \in \Phi^+, 0 \leq i \leq r-1},
\end{equation}
and relations
\begin{equation} \label{eq:scrArelations}
\begin{split}
X_\alpha X_\beta &= \zeta^{(\alpha,\beta)} X_\beta X_\alpha \quad \text{if $\alpha \prec \beta$}, \\
X_{p^i \ell \alpha} X_\beta &= X_\beta X_{p^i \ell \alpha}, \quad \text{and} \\
X_{p^i \ell \alpha} X_{p^j \ell \beta} &= X_{p^j \ell \beta} X_{p^i \ell \alpha}. 
\end{split}
\end{equation}
Applying \cite[Lemma 2.1]{Drupieski:2011a}, one sees that the associated graded algebra $\gr \Uzur$ is generated as an algebra by the symbols \eqref{eq:scrAgenerators}, subject to the relations \eqref{eq:scrArelations}, as well as the following additional relations:
\begin{equation}
X_\alpha^\ell = X_{p^i \ell \alpha}^p = 0 \quad \text{for each $\alpha \in \Phi^+$}.
\end{equation}
The algebra $\gr \Uzur$ inherits the structure of a $\Uzo$-module algebra, such that $X_\alpha$ has weight $-\alpha$ for $\Uzo$, and $X_{p^i\ell \alpha}$ has weight $-p^i\ell \alpha$ for $\Uzo$. Similarly, $\gr \Uzurp$ is also a $\Uzo$-module algebra, with generators $\set{Y_\alpha,Y_{p^i\ell \alpha}: \alpha \in \Phi^+, 0 \leq i \leq r-1}$ of weights $\alpha$ and $p^i\ell\alpha$ for $\Uzo$, respectively. The algebra $\gr \Uztr$ is canonically isomorphic to $\Uztr$.

\begin{lemma} \label{lemma:grUzgr}
The algebra $\gr \Uzgr$ is the smash product of $\Uztr$ and the tensor product of algebras $\gr \Uzur \otimes \gr \Uzurp$, that is, $\gr \Uzgr \cong (\gr \Uzur \otimes \gr \Uzurp) \# \Uztr$. The left action of $\Uztr$ on $\gr \Uzur \otimes\gr \Uzurp$ is induced by the left adjoint action of $\Uzo$ on $\Uzur$ and $\Uzurp$.
\end{lemma}

\subsection{Finite generation for nilpotent and Borel subalgebras} \label{subsection:fgborel}

Define the algebra $\Lambda_{\zeta,r}^\bullet$ to be the graded algebra with generators
\[
\set{x_\alpha,x_{p^i\ell \alpha}: \alpha \in \Phi^+,0 \leq i \leq r-1},
\]
each of graded degree 1, subject to the following relations:
\begin{align}
x_\alpha x_\beta + \zeta^{-(\alpha,\beta)}x_\beta x_\alpha &= 0 \quad \text{if} \quad \alpha \prec \beta, \label{eq:firstlambdarelation} \\
x_{p^i \ell \alpha}x_\beta + x_\beta x_{p^i \ell \alpha} &= 0, \\
x_{p^i \ell \alpha} x_{p^j \ell \beta} + x_{p^j \ell \beta}x_{p^i \ell \alpha} &= 0, \\
\text{and} \quad x_\alpha^2= x_{p^i \ell \alpha}^2 &= 0. \label{eq:lastlambdarelation}
\end{align}
Then $\Lambda_{\zeta,r}^\bullet$ is a left $U_\zeta^0$-module algebra by assigning weight $\alpha$ to $x_\alpha$ and weight $p^i\ell \alpha$ to $x_{p^i \ell \alpha}$.

\begin{lemma} \label{lemma:scrAcohomology}
There exists a graded $U_\zeta^0$-algebra isomorphism $\opH^\bullet(\scrA,k) \cong \Lambda_{\zeta,r}^\bullet$.
\end{lemma}

\begin{proof}
See \cite[Proposition 2.1]{Ginzburg:1993} or \cite[Theorem 4.1]{Mastnak:2010}.
\end{proof}

We now follow the strategy of \cite[\S 2.4]{Ginzburg:1993} to compute the structure of the cohomology ring $\opH^\bullet(\gr \Uzur,k)$. While we could use results of Mastnak, Pevtsova, Schauenburg and Witherspoon \cite[Theorem 4.1]{Mastnak:2010} to compute $\opH^\bullet(\gr \Uzur,k)$ directly, the extra information we obtain as a result of the inductive approach below will enable us to study the cohomology of the ungraded algebra $\Uzur$.

Enumerate the indeterminates in \eqref{eq:scrAgenerators} as $X_1,X_2,\ldots,X_m$. If $X_j = X_\alpha$ for some $\alpha \in \Phi^+$, set $X_j^\epsilon = X_j^\ell$; otherwise, set $X_j^\epsilon = X_j^p$. For $1 \leq j \leq m$, let $R_j$ be the vector subspace of $\scrA$ spanned by the elements $X_1^\epsilon,X_2^\epsilon,\ldots,X_j^\epsilon$, and let $\scrZ_j$ be the central subalgebra of $\scrA$ generated by $R_j$. Set $\fraku_j = \scrA//\scrZ_j$. Then $\fraku_0=\scrA$, $\fraku_{j+1} = \fraku_j//(X_{j+1}^\epsilon)$, and $\fraku_m=\gr \Uzur$.

\begin{proposition} \label{proposition:inductivealgebraiso}
For each $0 \leq j \leq m$, there is a graded $U_\zeta^0$-module algebra isomorphism
\[ \opH^\bullet(\fraku_j,k) \cong \Lambda_{\zeta,r}^\bullet \otimes S^\bullet(R_j^*), \]
where the elements of $S^1(R_j^*) = R_j^* = \Hom_k(R_j,k)$ are assigned graded degree two.
\end{proposition}

\begin{proof}
We proceed by induction on $j$, following the strategy of \cite[\S 2.4]{Ginzburg:1993}. In fact, with the exception of Step 4, the proof is formally the same as that given by Ginzburg and Kumar for the case $r=0$.

If $j = 0$, the proposition reduces to Lemma \ref{lemma:scrAcohomology}, so assume by induction that the proposition is true for all $0 \leq i \leq j$. Let $A = A_{j+1}$ be the central subalgebra of $\fraku_j$ generated by $X_{j+1}^\epsilon$. Then $\fraku_j//A \cong \fraku_{j+1}$, and $\fraku_j$ is free over $A$. Then by Theorem \ref{theorem:LHSspecseqofalgebras}, there exists a spectral sequence of $\Uzo$-module algebras satisfying
\begin{equation} \label{eq:ujspecseq}
E_2^{a,b} = \opH^a(\fraku_{j+1},k) \otimes \opH^b(A,k) \Rightarrow \opH^{a+b}(\fraku_j,k).
\end{equation}

\noindent \textbf{Step 1.} For each $a \geq 0$, the restriction homomorphism $r_a:\opH^a(\fraku_{j+1},k) \rightarrow \opH^a(\fraku_j,k)$ is surjective. The map $r_1: \opH^1(\fraku_{j+1},k) \rightarrow \opH^1(\fraku_j,k)$ is an isomorphism.

\begin{proof}
By the induction hypothesis, $\opH^\bullet(\fraku_j,k)$ is generated by elements of degree $\leq 2$, so it suffices to prove the surjectivity of $r_1$ and $r_2$. The map $r_1$ is an isomorphism because, for any augmented algebra $B$, $\opH^1(B,k) \cong (B_+/B_+^2)^*$. The surjectivity of $r_2$ follows from \cite[Lemma 2.10]{Gordon:2000}.
\end{proof}

\noindent \textbf{Step 2.} In \eqref{eq:ujspecseq}, $E_\infty^{a,b}=0$ for all $b>0$.

\begin{proof}
There exists a commutative diagram
\begin{equation} \label{eq:commdiagram}
\xymatrix{
E_2^{a-2,1} \ar[r]^{d_2} & E_2^{a,0} \ar@{->>}[r] \ar@{=}[d] & E_\infty^{a,0} \ar@{^{(}->}[r] & \opH^a(\fraku_j,k) \\
&  \opH^a(\fraku_{j+1},k) \ar[urr]_{r_a} & &
}
\end{equation}
By Step 1, $r_a$ is surjective. Then the inclusion $E_\infty^{a,0} \hookrightarrow \opH^a(\fraku_j,k)$ must be an isomorphism, hence $E_\infty^{a,b}=0$ for all $b>0$ (because the spectral sequence converges to $\opH^\bullet(\fraku_j,k)$).
\end{proof}

\noindent \textbf{Step 3.} Choose $0 \neq v \in \im(d_2^{0,1}) \subseteq \opH^2(\fraku_{j+1},k)$. Such $v$ exists because $E_\infty^{0,1}=0$ and $\dim \opH^1(A,k)=1$. Then the kernel of the restriction map $r:\opH^\bullet(\fraku_{j+1},k) \rightarrow \opH^\bullet(\fraku_j,k)$ is generated by $v$.

\begin{proof}
From \eqref{eq:commdiagram} and the isomorphism $E_\infty^{a,0} \cong \opH^a(\fraku_j,k)$, we conclude that the kernel of the restriction map $r:\opH^\bullet(\fraku_{j+1},k) \rightarrow \opH^\bullet(\fraku_j,k)$ is equal to the image of $E_2^{\bullet,1}$ under the differential $d_2$. Choose $y \in \opH^1(A,k)$ with $d_2^{0,1}(y)=v$. Now an arbitrary element of $E_2^{n,1}$ can be written in the form $x \otimes y$ for some $x \in \opH^n(\fraku_{j+1},k)$. Then $d_2(x \otimes y) = d_2(x) \otimes y + (-1)^n x \cdot d_2(y) = (-1)^n x \cdot v$ is an element of the two-sided ideal in $\opH^\bullet(\fraku_{j+1},k)$ generated by $v$.
\end{proof}

\noindent \textbf{Step 4.} The algebra homomorphism $r:\opH^\bullet(\fraku_{j+1},k) \rightarrow \opH^\bullet(\fraku_j,k)$ admits a graded algebra splitting that commutes with the action of $U_\zeta^0$.

\begin{proof}
By induction, there exists a $U_\zeta^0$-module algebra isomorphism $\opH^\bullet(\fraku_j,k) \cong \Lambda_{\zeta,r}^\bullet \otimes S^\bullet(R_j^*)$. To prove the claim, we must lift the generators for $\Lambda_{\zeta,r}^\bullet$ and $S^\bullet(R_j^*)$ to $\opH^\bullet(\fraku_{j+1},k)$, and check that the relations among the generators are preserved.

To the pairs $(\scrA,\scrZ_j)$ and $(\scrA,\scrZ_{j+1})$, there exist LHS spectral sequences as in Theorem \ref{theorem:LHSspecseqofalgebras}. The natural morphism $(\scrA,\scrZ_j) \rightarrow (\scrA,\scrZ_{j+1})$ induced by the inclusion $\scrZ_j \subset \scrZ_{j+1}$ induces a morphism of spectral sequence, the low degree terms of which form the following commutative diagram of $\Uzo$-modules:
\[
\xymatrix{
R_{j+1}^* \ar@{=}[r] \ar@{->>}[d] & \opH^1(\scrZ_{j+1},k) \ar@{->>}[d] \ar[r]^{d'} & \opH^2(\scrA//\scrZ_{j+1},k) \ar@{->>}[d]^r \ar@{=}[r] & \opH^2(\fraku_{j+1},k) \ar@{->>}[d]^r \\
R_j^* \ar@{=}[r] & \opH^1(\scrZ_j,k) \ar[r]^{d''} & \opH^2(\scrA//\scrZ_j,k) \ar@{=}[r] & \opH^2(\fraku_j,k)
}
\]
The horizontal maps $d',d''$ are the differentials of the appropriate spectral sequences, and the projection $R_{j+1}^* \twoheadrightarrow R_j^*$ is the restriction of functions. Also, the image of $R_j^*=\opH^1(\scrZ_j,k)$ in $\opH^2(\fraku_j,k)$ under $d''$ identifies with the subspace $S^1(R_j^*) \subset S^\bullet(R_j^*) \subset \opH^\bullet(\fraku_j,k)$. Now any $U_\zeta^0$-module splitting of the projection $R_{j+1}^* \twoheadrightarrow R_j^*$ provides a lifting of the generators of $S^\bullet(R_j^*)$ to $\opH^\bullet(\fraku_{j+1},k)$. The lifted generators have central image in $\opH^\bullet(\fraku_{j+1},k)$ by \cite[Corollary 5.3]{Ginzburg:1993}.

Next, consider the generators $x_1,\ldots,x_m$ of $\Lambda_{\zeta,r}^\bullet$ as elements of $\opH^1(\fraku_j,k)$. By Step 1, the restriction map $r_1:\opH^1(\fraku_{j+1},k) \rightarrow \opH^1(\fraku_j,k)$ is a $U_\zeta^0$-module isomorphism. We use the inverse map $(r_1)^{-1}: \opH^1(\fraku_j,k) \rightarrow \opH^1(\fraku_{j+1},k)$ to transfer the generators of $\opH^1(\fraku_j,k)$ to $\opH^1(\fraku_{j+1},k)$. Set $\wt{x}_i = (r_1)^{-1}(x_i)$. To prove the claim of Step 4, it now suffices to show that the elements $\wt{x}_1,\ldots,\wt{x}_m \in \opH^\bullet(\fraku_{j+1},k)$ satisfy the relations (\ref{eq:firstlambdarelation}--\ref{eq:lastlambdarelation}).

Consider, for example, the relation (\ref{eq:firstlambdarelation}). If $\alpha \prec \beta$, then
\[ \textstyle
r \left( \wt{x}_\alpha \wt{x}_\beta + \zeta^{-(\alpha,\beta)} \wt{x}_\beta \wt{x}_\alpha \right) = x_\alpha x_\beta + \zeta^{-(\alpha,\beta)} x_\beta x_\alpha = 0,
\]
hence
\begin{equation} \label{eq:tilderelation}
\wt{x}_\alpha \wt{x}_\beta + \zeta^{-(\alpha,\beta)} \wt{x}_\beta \wt{x}_\alpha = cv
\end{equation}
for some $c \in k$ by Step 3. So we must show $c=0$. Here is where our argument deviates from that of \cite[\S 2.4]{Ginzburg:1993}: Since the algebra $\scrA$ is defined in terms of homogeneous relations on the independent generators $X_1,\ldots,X_m$, we can define an action of any $m$-dimensional algebraic torus $T^m = (k^\times)^m$ on $\scrA$ by declaring the generator $X_i$ to have weight $-\chi_i$ for $T^m$, where $\chi_i:T^m \rightarrow k^\times$ denotes the $i$-th coordinate function. This induces an action of $T^m$ on $\Lambda_{\zeta,r}^\bullet \cong \opH^\bullet(\scrA,k)$ such that $x_i$ has weight $\chi_i$ for $T^m$. Now suppose $x_\alpha = x_a$ and $x_\beta = x_b$ with $1 \leq a,b \leq m$. Then the left side of \eqref{eq:tilderelation} has weight $\chi_a+\chi_b$ for $T^m$, while the right side has weight $\epsilon \cdot \chi_{j+1}$ for $T^m$, where as before we set $\epsilon=\ell$ if $X_{j+1}=X_\gamma$ for some $\gamma \in \Phi^+$, and $\epsilon = p$ otherwise. Since $a \neq b$, and since $\ell$ and $p$ are odd, we must have $\chi_a+\chi_b \neq \epsilon \cdot \chi_{j+1}$. This forces $c=0$. The other relations among the $\wt{x}_i$ are proved in a similar manner.
\end{proof}

\noindent \textbf{Step 5.} The element $v$ introduced in Step 3 is not a zero-divisor in $\opH^\bullet(\fraku_{j+1},k)$.

\begin{proof}
By Step 2, $E_3^{a,1} \cong E_\infty^{a,1} = 0$, hence the differential $d_2^{a,1}$ in \eqref{eq:ujspecseq} must be injective. Now for $0 \neq x \otimes y \in \opH^a(\fraku_{j+1},k) \otimes \opH^1(A,k) = E_2^{a,1}$, we have
\[
0 \neq d_2(x \otimes y) = d_2(x) \cdot y + (-1)^a x \cdot d_2(y) = (-1)^a x \cdot v.
\qedhere \]
\end{proof}

\noindent The results of Steps 3--5 complete the proof of the proposition.
\end{proof}

The multiplicative filtration on $\Uzur$ induces a decreasing filtration on the cobar complex $\Cbul(\Uzur,k)$ computing $\opH^\bullet(\Uzur,k)$, and hence gives rise to a spectral sequence of algebras
\begin{equation} \label{eq:grUzurspecseq}
E_1^{i,j} = \opH^{i+j}(\gr \Uzur,k)_{(i)} \Rightarrow \opH^{i+j}(\Uzur,k),
\end{equation}
where the subscript on the $E_1$-term denotes the internal grading induced by the grading of $\gr \Uzur$. Recall the notation introduced just before Proposition \ref{proposition:inductivealgebraiso}. Set $R=R_m$ and $\scrZ = \scrZ_m$. Then $\opH^\bullet(\gr \Uzur,k) \cong \Lambda_{\zeta,r}^\bullet \otimes S^\bullet(R^*)$, where the space $R^* = S^1(R^*)$ is assigned graded degree two.

\begin{proposition} \label{proposition:Rpermanentcycles}
In \eqref{eq:grUzurspecseq}, the subspace $R^* = S^1(R^*)$ of $\opH^2(\gr \Uzur,k)$ consists of permanent cycles.
\end{proposition}

\begin{proof}
Consider the LHS spectral sequence for the pair $(\scrA,\scrZ)$  constructed in Theorem \ref{theorem:LHSspectralsequence}:
\begin{equation} \label{eq:LHSspecseqscrAscrZ}
E_2^{i,j} = \opH^i(\gr \Uzur,\opH^j(\scrZ,k)) \Rightarrow \opH^{i+j}(\scrA,k).
\end{equation}
It follows from the proof of Proposition \ref{proposition:inductivealgebraiso} that $R^* \subset \opH^2(\gr \Uzur,k)$ identifies with the image of of $E_2^{0,1} = \opH^1(\scrZ,k) \cong \Lambda^1(R^*) = R^*$ under the differential $d_2^{0,1}: E_2^{0,1} \rightarrow E_2^{2,0}$. Choose $j \in \set{1,\ldots,m}$, and let $x_j \in R^* \cong \opH^1(\scrZ,k)$ be dual to $X_j^\epsilon \in R$. Then a cocycle representative for $x_j$ in $C^1(\scrZ,k) \cong \Hom_k(\scrZ_+,k)$ is the linear map dual to the vector $X_j^\epsilon \in \scrZ_+$. We will determine an explicit cocycle representative in $C^2(\gr \Uzur,k)$ for $d_2^{0,1}(x_j) \in \opH^2(\gr \Uzur,k)$, and then show that this cocycle representative is induced by a cocycle in $C^2(\Uzur,k)$. Since cocycles in $\Cbul(\Uzur,k)$ become permanent cycles in \eqref{eq:grUzurspecseq}, this will prove the proposition.

The spectral sequence \eqref{eq:LHSspecseqscrAscrZ} is the spectral sequence arising from the column-wise filtration on the double complex $C = C^{\bullet,\bullet} = \Hom_{\scrA}(P^\bullet,Q_\bullet)$, where $P^\bullet = \Bbul(\scrA//\scrZ)$ is the left bar resolution of $\scrA//\scrZ$, and $Q_\bullet = Q_\bullet(k)$ is the coinduced resolution of the $\scrA$-module $k$. Write $d_h$ and $d_v$ for the horizontal and vertical differentials on $C$, induced by the differentials for $P^\bullet$ and $Q_\bullet$, respectively. Then the total differential $d$ of the total complex $\Tot(C)$ is $d = d_h + (-1)^i d_v$ (i.e., the vertical differential along the $i$-th column is replaced by $(-1)^i d_v$).

By \cite[Theorem 2.7]{McCleary:2001}, the $E_2^{0,1}$ term of \eqref{eq:LHSspecseqscrAscrZ} is represented by elements $(x,y) \in C^{0,1} \oplus C^{1,0}$ such that $d_v(x)=0$ and $d_h(x)-d_v(y)=0$, while $E_2^{2,0}$ is represented by elements of $(\ker d) \cap C^{2,0}$. The differential $d_2^{0,1}: E_2^{0,1} \rightarrow E_2^{2,0}$ is then induced by the total differential of $\Tot(C)$. We claim that $x_j \in E_2^{0,1}$ is represented by the element $f_{0,1} \oplus f_{1,0} \in C^{0,1} \oplus C^{1,0}$, and that $d_2^{0,1}(x_j) \in E_2^{2,0}$ is represented by $f_{2,0} \in C^{2,0}$, where the elements $f_{0,1}$, $f_{1,0}$, and $f_{2,0}$ are defined as follows:
\begin{itemize}
\item $f_{0,1} \in \Hom_k(\scrA_+ \otimes \scrA//\scrZ,k) \cong C^{0,1}$ evaluates to $1$ on the monomial $X_j^a \otimes X_j^b$ if $a \geq 1$ and $a + b = \epsilon$, and evaluates to zero on all other monomial basis elements of $\scrA_+ \otimes \scrA//\scrZ$.

\item $f_{1,0} \in \Hom_k((\scrA//\scrZ)_+,\Hom_k(\scrA//\scrZ,k)) \cong C^{1,0}$ sends the monomial $X_j^a$ ($1 \leq a < \varepsilon$) to the linear map $g_{1,0,a} \in \Hom_k(\scrA//\scrZ,k)$, and evaluates to zero on all other monomial basis elements of $(\scrA//\scrZ)_+$. For $1 \leq a < \epsilon$, the linear map $g_{1,0,a} \in \Hom_k(\scrA//\scrZ,k)$ is the function dual to the basis vector $X_j^{\epsilon - a} \in \scrA//\scrZ$.

\item $f_{2,0} \in \Hom_k((\scrA//\scrZ)_+^{\otimes 2},\Hom_k(\scrA//\scrZ,k))$ evaluates to zero on the vector $X_j^a \otimes X_j^b$ ($1 \leq a,b < \epsilon$) if $a+b \neq \epsilon$, evaluates to the counit $\varepsilon: \scrA//\scrZ \rightarrow k$ on the vector $X_j^a \otimes X_j^b$ if $a+b = \epsilon$, and evaluates to zero on all other monomial basis elements in $(\scrA//\scrZ)_+^{\otimes 2}$.
\end{itemize}

One can check from the definitions that $d_v(f_{0,1}) = 0$, $d_h(f_{0,1}) - d_v(f_{1,0}) = 0$, and $f_{2,0} = d_h(f_{1,0}) = d(f_{0,1} \oplus f_{1,0})$. Then $f_{2,0} \in (\ker d) \cap C^{2,0}$. In particular, $f_{2,0} \in \ker d_v|_{C^{2,0}}$. Also, the projectivity of $P^2$ as a module for $\scrA//\scrZ$ implies that
\begin{equation} \textstyle
\ker d_v|_{C^{2,0}} \cong \Hom_{\scrA//\scrZ}\left( P^2, \ker \left( d_v: Q_0^{\scrZ} \rightarrow Q_1^{\scrZ} \right) \right) \cong C^2(\scrA//\scrZ,k),
\end{equation}
because the kernel of $d_v: Q_0^{\scrZ} \rightarrow Q_1^{\scrZ}$ is the one-dimensional subspace of $\Hom_k(\scrA,k) \cong \Hom_\scrA(\scrA \otimes \scrA,k) = Q_0(k)$ spanned by the counit $\varepsilon: \scrA \rightarrow k$.

We must check that the image of $f_{0,1} \oplus f_{1,0}$ in $E_2^{0,1}$ is $x_j$. Since $\scrA$ is flat (in fact, free) as a right $\scrZ$-module, the cohomology ring $\opH^\bullet(\scrZ,k)$ may be computed by applying the functor $-^{\scrZ}$ to either an $\scrA$-injective or a $\scrZ$-injective resolution of $k$. Write $\Qbul(k)$ for the coinduced resolution of $k$ as an $\scrA$-module, and write $\Qbul'(k)$ for the coinduced resolution of $k$ as a $\scrZ$-module. Then $E_1^{0,\bullet} \cong \Qbul(k)^{\scrZ}$, and $d_1^{0,\bullet}: E_1^{0,\bullet} \rightarrow E_1^{0,\bullet+1}$ is induced by the differential of $\Qbul(k)$. On the other hand, $\Qbul'(k)^{\scrZ} \cong \Cbul(\scrZ,k)$ as complexes, and under the restriction homomorphism $\Qbul(k)^{\scrZ} \rightarrow \Qbul'(k)^{\scrZ}$, $f_{0,1}$ is mapped to the cocycle $f \in C^1(\scrZ,k)$ that is dual to the vector $X_j^\epsilon$. This proves the claim for $f_{0,1} \oplus f_{1,0}$.

We now define a cocycle $f_2 \in C^2(\Uzur,k)$ that induces $f_{2,0} \in \ker d_v|_{C^{2,0}} \cong C^2(\scrA//\scrZ,k)$. If $X_j = X_\alpha$ for some $\alpha \in \Phi^+$, set $F_j = F_\alpha$. If $X_j = X_{p^i \ell \alpha}$ for some $\alpha \in \Phi^+$ and some $i \geq 0$, set $F_j = F_\alpha^{(p^i\ell)}$. Now define $f_2 \in C^2(\Uzur) \cong \Hom_k(\Uzur_+^{\otimes 2},k)$ to be the linear map that evaluates to 1 on the monomial $F_j^a \otimes F_j^b$ if $a,b \geq 1$ and $a + b = \epsilon$, and evaluates to zero on all other monomial basis elements of $\Uzur_+^{\otimes 2}$. We claim that $f_2$ is a cocycle in $C^\bullet(\Uzur)$. Indeed, let $\delta$ be the differential of the cobar complex $\Cbul(\Uzur)$, and let $F^{(\mbf{a})},F^{(\mbf{b})},F^{(\mbf{c})} \in \Uzur_+$ ($\mbf{a},\mbf{b},\mbf{c} \in \N^N$) be monomial basis vectors for $\Uzur$. Consider $c:= (\delta f_2)(F^{(\mbf{a})} \otimes F^{(\mbf{b})} \otimes F^{(\mbf{c})}) = f_2(F^{(\mbf{a})} \otimes F^{(\mbf{b})} F^{(\mbf{c})} - F^{(\mbf{a})} F^{(\mbf{b})} \otimes F^{(\mbf{c})})$. If $c \neq 0$, then by the definition of $f_2$ we must have (up to a unit in $k$) $F^{(\mbf{a})} = F_j^a$ and $F^{(\mbf{c})} = F_j^c$ for some $a,c \geq 1$. If $c \neq 0$, then \cite[Lemma 2.1]{Drupieski:2011a} further implies that, up to a unit in $k$, $F^{(\mbf{b})} = F_j^b$ for some $b \geq 1$. Now $(\delta f_2)(F_j^a \otimes F_j^b \otimes F_j^c) = f_2(F_j^a \otimes F_j^{b+c} - F_j^{a+b} \otimes F_j^c)$, and this evaluates to zero for all possible combinations of $a,b,c$. Thus $f_2 \in \ker \delta$. Finally, it is clear that $f_2$ induces the map $f_{2,0}$, hence that $x_j \in R^* \subset \opH^2(\gr \Uzur,k)$ is a permanent cycle in \eqref{eq:grUzurspecseq}.
\end{proof}

\begin{remark}
Bendel, Nakano, Parshall and Pillen prove the previous proposition in the special case $r=0$ by a weight argument, though they must assume $\ell > 3$ whenever $\Phi$ has type $B$ or $C$; see \cite[Proposition 6.2.2]{Bendel:2011}. Our proof does not require the extra assumption on $\ell$.
\end{remark}

\begin{corollary} \label{corollary:Uzurcohomologyfg}
The cohomology ring $\opH^\bullet(\Uzur,k)$ is finitely generated.
\end{corollary}

\begin{proof}
This follows from Proposition \ref{proposition:Rpermanentcycles} and \cite[Lemma 2.5]{Mastnak:2010}.
\end{proof}

Let $A$ be an arbitrary augmented algebra over $k$, and let $M$ be a left $A$-module. Then $\opH^\bullet(A,M) = \Ext_A^\bullet(k,M)$ is a right graded $\opH^\bullet(A,k) = \Ext_A^\bullet(k,k)$-module via Yoneda composition of extensions.

\begin{theorem}
Set $A$ equal to either $\Uzur$ or $\Uzbr$, and let $M$ be a finite-dimensional $A$-module. Then $\opH^\bullet(A,M)$ is a finite module for the Noetherian algebra $\opH^\bullet(A,k)$.
\end{theorem}

\begin{proof}
First suppose $A = \Uzur$. Then $\opH^\bullet(A,k)$ is Noetherian by Corollary \ref{corollary:Uzurcohomologyfg}. By \cite[\S 3.1]{Drupieski:2011a}, the trivial module $k$ is the unique simple module for $A$. Then there exists a filtration of $M$ by submodules $M = M_0 \supset M_1 \supset \cdots \supset M_s \supset M_{s+1}=0$ such that, for each $0 \leq i \leq s$, $M_i/M_{i+1} \cong k$. Now by a standard argument using induction on the dimension of $M$ and the long exact sequence in cohomology, $\opH^\bullet(A,M)$ is a finite module over the Noetherian algebra $\opH^\bullet(A,k)$.

Now suppose $A = \Uzbr$. Then
\[
\opH^\bullet(A,M) = \opH^\bullet(\Uzur,M)^{\Uztr} \quad \text{and} \quad \opH^\bullet(A,k) = \opH^\bullet(\Uzur,k)^{\Uztr},
\]
as can be seen by considering the LHS spectral sequence for the algebra $\Uzbr$ and its normal subalgebra $\Uztr$, and by using the fact that all finite-dimensional representations of $\Uztr$ are semisimple \cite[Lemma A 3.4]{Du:1994}. Since the right action of $\opH^\bullet(\Uzur,k)$ on $\opH^\bullet(\Uzur,M)$ is a $\Uztr$-module homomorphism, and since $\Uztr$ acts completely reducibly on $\opH^\bullet(\Uzur,k)$ and $\opH^\bullet(\Uzur,M)$, the theorem then follows from \cite[Lemma 1.13]{Friedlander:1986}.
\end{proof}

\subsection{Finite complexity} \label{subsection:finitecomplexity}

Let $V = \bigoplus_{n \in \N} V_n$ be a graded vector space having finite-dimensional homogeneous components. Define the rate of growth $\gamma(V)$ to be the least non-negative integer $c \in \N$ such that there exists $b \in \R$ for which $\dim V_n \leq bn^{c-1}$ for all $n \in \N$. If no such $c$ exists, set $\gamma(V) = \infty$. Then, for example, if $V_n$ is the vector space of homogeneous degree $n$ polynomials in the polynomial ring $k[x_1,\ldots,x_s]$, then $\gamma(V) = s$.

Let $A$ be a finite-dimensional algebra. Recall that the \emph{complexity} $\cx_A(M)$ of an $A$-module $M$ is defined as the rate of growth $\gamma(P^\bullet)$ of a minimal projective resolution $P^\bullet$ for $M$. In particular, $\cx_A(k) = \gamma(\opH^\bullet(A,k))$. If $\opH^\bullet(A,k)$ is finitely-generated, then $\cx_A(k) < \infty$. If $A$ is a Hopf algebra, then $\cx_A(M) \leq \cx_A(k)$ for all finite-dimensional $A$-modules $M$ \cite[Proposition 2.1]{Feldvoss:2010}.

Friedlander and Suslin \cite{Friedlander:1997} proved that the cohomology ring $\opH^\bullet(G_r,k)$ for the Frobenius kernel $G_r$ of $G$ is finitely-generated. They did this by embedding the group $G$ into $GL_n$ for some $n \in \N$, and then reducing to the case of Frobenius kernels for $GL_n$. Unfortunately, such a strategy for quantum groups is problematic, because there is no obvious embedding of an arbitrary quantum group into one for which the underlying root system is of type $A$. Still, the following theorem provides circumstantial evidence for the finite generation of $\opH^\bullet(\Uzgr),k)$, by showing that the complexity of the trivial module is finite.

\begin{theorem} \label{theorem:finitecomplexity}
Let $M$ be a finite-dimensional $\Uzgr$-module. Then
\[
\cx_{\Uzgr}(M) \leq (r+1)(\dim(\g) - \rank(\g)).
\]
If $\ell > h$ and if $p = \chr(k)$ is good for $\Phi$, then $\cx_{\Uzgr}(M) \leq (r+1) \cdot \cx_{\uzg}(k)$.
\end{theorem}

\begin{proof}
By \cite[Proposition 2.1]{Feldvoss:2010}, it suffices to treat the case $M = k$. The multiplicative filtration on $\Uzgr$ described in Lemma \ref{lemma:Nfiltration} gives rise to the spectral sequence
\[
E_1^{i,j} = \opH^{i+j}(\gr \Uzgr,k)_{(i)} \Rightarrow \opH^{i+j}(\Uzgr,k).
\]
Then $\cx_{\Uzgr}(k) = \gamma(\opH^\bullet(\Uzgr,k)) \leq \gamma(\opH^\bullet(\gr \Uzgr,k))$. The algebra $\Uztr$ is semisimple over $k$, and $\gr \Uzgr = (\gr \Uzur \otimes \gr \Uzurp) \# \Uztr$ by Lemma \ref{lemma:grUzgr}, hence
\[
\opH^\bullet(\gr \Uzgr,k) = \opH^\bullet(\gr \Uzur \otimes \gr \Uzurp,k)^{\Uztr}.
\]
Then $\gamma(\opH^\bullet(\gr \Uzgr,k)) \leq \gamma(\opH^\bullet(\gr \Uzur \otimes \gr \Uzurp,k))$. By Proposition \ref{proposition:inductivealgebraiso},
\[
\opH^\bullet(\gr \Uzur \otimes \gr \Uzurp,k) \cong \opH^\bullet(\gr \Uzur,k) \otimes \opH^\bullet(\gr \Uzurp,k)
\]
is finitely generated over a polynomial ring in $(r+1)(\dim(\g) - \rank(\g))$ indeterminates. Finally, if $\ell > h$ and if $p$ is good for $\Phi$, then $\cx_{\uzg}(k) = \dim(\g) - \rank(\g)$.
\end{proof}

\begin{remark}
One can compare Theorem \ref{theorem:finitecomplexity} to a result of Nakano \cite[Theorem 2.4]{Nakano:1995}, which states that $\cx_{G_r}(k) \leq r \cdot \cx_{G_1}(k)$. If $p > h$, then $\cx_{G_1}(k) = \dim(\g) - \rank(\g)$.
\end{remark}

\section*{Acknowledgements}

This paper is based in part on the author's PhD thesis. As a graduate student, the author was supported in part by NSF grant DMS-0701116. While preparing this paper, the author was supported in part by NSF VIGRE grant DMS-0738586.


\providecommand{\bysame}{\leavevmode\hbox to3em{\hrulefill}\thinspace}

\end{document}